\documentclass[11pt,leqno]{article}
\usepackage{amssymb}
\usepackage{amsfonts}
\usepackage{stmaryrd}
\usepackage{amsmath, amsthm, amsfonts, amssymb}
\usepackage{mathrsfs}
\usepackage{bbm}
\usepackage{esint}
\usepackage{mathrsfs}
\usepackage{amssymb,amsmath,amsfonts,amsthm}
\usepackage[colorlinks, citecolor=blue,  hypertexnames=false]{hyperref}

\allowdisplaybreaks
\newfam\msbfam
\font\tenmsb=msbm10    \textfont\msbfam=\tenmsb
\font\sevenmsb=msbm7 \scriptfont\msbfam=\sevenmsb
\font\fivemsb=msbm5 \scriptscriptfont\msbfam=\fivemsb

\newfam\bigfam
\font\tenbig=msbm10 scaled \magstep2   \textfont\bigfam=\tenbig
\font\sevenbig=msbm7 scaled \magstep2 \scriptfont\bigfam=\sevenbig
\font\fivebig=msbm5 scaled \magstep2
\scriptscriptfont\bigfam=\fivebig

\def\aaverage{{\mathchoice {\kern1ex\vcenter{\hrule height.4pt width 12pt
depth0pt} \kerQ-16pt} {\kern1ex\vcenter{\hrule height.4pt width 9pt
depth0pt} \kerQ-12pt} {} {} }}

\def\vbar{\mathchoice{\vrule height6.3ptdepth-.5ptwidth.8pt\kerQ-.8pt}
   {\vrule height6.3ptdepth-.5ptwidth.8pt\kerQ-.8pt}
   {\vrule height4.1ptdepth-.35ptwidth.6pt\kerQ-.6pt}
   {\vrule height4.1ptdepth-.25ptwidth.5pt\kerQ-.5pt}}
\def\fudge{\mathchoice{}{}{\mkern.5mu}{\mkern.8mu}}
\def\bbc#1#2{{\rm \mkern#2mu\vbar\mkerQ-#2mu#1}}
\def\bbb#1{{\rm I\mkerQ-4.5mu #1}}
\def\bba#1#2{{\rm #1\mkerQ-#2mu\fudge #1}}
\def\bb#1{{\count4=`#1 \advance\count4by-64 \ifcase\count4\or\bba A{11.5}\or
   \bbb B\or\bbc C{5}\or\bbb D\or\bbb E\or\bbb F \or\bbc G{5}\or\bbb H\or
   \bbb I\or\bbc J{3}\or\bbb K\or\bbb L \or\bbb M\or\bbb N\or\bbc O{5} \or
   \bbb P\or\bbc B{5}\or\bbb R\or\bba S{8}\or\bba T{10.5}\or\bbc U{5}\or
   \bba V{12}
\or\bba W{16.5}\or\bba X{11}\or\bba Y{11.7}\or\bba Z{7.5}\fi}}

\def\be{\begin{eqnarray*}}
\def\ee{\end{eqnarray*}}
\def\be{\begin{equation}}
\def\ee{\end{equation}}

\newtheorem{thm}{Theorem}[section]
\newtheorem{lem}{Lemma}[section]
\newtheorem{prop}{Proposition}[section]
\newtheorem{rem}{Remark}[section]
\newtheorem{cor}{Corollary}[section]
\newtheorem{defn}{Definition}[section]

\allowdisplaybreaks

\begin{document}

%\\\\\\\\\\\\\\title\\\\\\\\\\\\\\
\title{\bf Real-Variable Theory of Local Variable Hardy Spaces}
\author{{\bf  Jian Tan}
\\\footnotesize\scriptsize  \textit{School of Science, Nanjing University of Posts and Telecommunications, Nanjing 210023, People's Republic of China}
\\\footnotesize\scriptsize  \textit{Department of Mathematics, Nanjing University, Nanjing 210093, People's Republic of China}
\\\footnotesize\textit{ E-mail:\quad  tanjian89@126.com}}

\date{}
\maketitle \vspace{0.2cm}

\noindent{{\bf Abstract}} \quad
In this paper, we give a complete real-variable theory of local variable Hardy spaces.
First, we present various real-variable characterizations in terms of several local maximal functions. Next, the new atomic and the finite atomic decomposition for
the local variable Hardy spaces are established.
As an application, we also introduce the local variable 
Campanato space which is showed to
be the dual space of the local variable Hardy spaces. Analogous to the homogeneous case, some equivalent definitions of the dual of local variable Hardy spaces are also considered.
Finally, we show that the boundedness of inhomogeneous Calder\'on-Zygmund singular integrals and local fractional integrals on local variable Hardy spaces and their duals.
\medskip

\noindent{{\bf MR(2010) Subject Classification}} \quad  Primary 42B30; Secondary
42B25, 42B35, 46E30

\noindent{{\bf Keywords}} \quad Local Hardy space, atom,
variable exponent analysis, local BMO-type space,
inhomogeneous Calder\'on-Zygmund singular integrals,
local fractional integrals.

%\\\\\\\\\\\\\\\Introduction\\\\\\\\\\\\

\section{Introduction}
The real-variable theory of classical global Hardy spaces $H^p(\mathbb R^n)$ in the Euclidean spaces developed by Stein and Weiss (\cite{SW}) and systematically
developed by Fefferman and Stein (\cite{FS}). It is well known that $H^p(\mathbb R^n)$ with $0<p\le1$ is a good substitute of the Lebesgue space $L^p$ when studying the boundedness of classical operators in harmonic analysis. Moreover, 
the atomic characterizations of $H^p(\mathbb R^n)$ is a very improtant tool for the study of function spaces and
the operators acting on these spaces. 
The atomic charaterization of $H^p(\mathbb R^n)$ in one dimension is given by 
Coifman (\cite{C}) in 1974 and later was extended to higher dimensions by Latter (\cite{L}).

However, it is pointed out that $H^p(\mathbb R^n)$ is well suited only to the Fourier analysis, but is not stable under multiplication by Schwartz class.
To circumvent the drawbacks,  in 1979, Goldberg (\cite{Go}) introduced the theory of local Hardy space $h^{p}(\mathbb R^n)$ in the Euclidean spaces, which plays an important role in various fields of
analysis and partial differential equations. Particularly,
Goldberg (\cite{Go}) obtained the atomic decomposition
characterizations of $h^{p}(\mathbb R^n)$, introduced the dual spaces for $h^{p}(\mathbb R^n)$, and showed that pseudo-differential operators of order zero are bounded on local Hardy spaces $h^p(\mathbb R^n)$ for $0<p<1$. Differently from $H^p(\mathbb R^n)$, when considering the
atom decompositions of $h^{p}(\mathbb R^n)$, the cancellation properties is needed only for the atoms with small atoms. 
From then on, the theory of local Hardy spaces have 
attracted many attentions by many researchers. 
In 1981, Bui (\cite{Bu}) investigated the weighted local Hardy spaces $h_w^{p}(\mathbb R^n)$   
with $w\in A_\infty(\mathbb R^n)$. 
In 2012, Tang (\cite{Tang}) introduced the weighted local Hardy spaces $h_w^{p}(\mathbb R^n)$ associated 
with local weights which include the classical Muckenhoupt weights. In \cite{Tang}, the atomic decomposition characterizations of these $h_w^{p}(\mathbb R^n)$ is established. As applications, Tang 
also obtain that some strongly singular integrals
and pseudo-differential operators  with the classical symbols on these spaces.
Meanwhile, D. Yang and S. Yang (\cite{YY}) studied 
the weighted local Orlicz-Hardy spaces in terms of the local grand maximal function and weighted atomic local Orlicz-Hardy spaces. They established the atomic decomposition characterizations and introduced the dual
spaces for those spaces. As applications, they also showed some boundedness for the local Riesz transform, the local fractional integral operator, and pseudo-differential operators with the classical symbols on those spaces. 
In \cite{YLK,YY1}, Yang et al. studied local Musielak-Orlicz Hardy spaces.
Lately, Sawano et al. (\cite{SHYY}) 
introduced local Hardy spaces for ball quasi-Banach function spaces.
Very recently, He et al. (\cite{HYY}) 
introduced the local Hardy spaces on spaces of homogeneous type via local grand maximal functions, obtained the finite atomic characterizations
for the local Hardy spaces, and also give the dual spaces of these spaces.

On the other hand, the study of variable Hardy spaces $H^{p(\cdot)}(\mathbb R^n)$ is inspired by the Lebesgue spaces with variable exponents $L^{p(\cdot)}(\mathbb R^n)$, which gain the attentions of many researchers.
The theory of variable Hardy spaces was developed 
independently by
Nakai and Sawano (\cite{NS}), Cruz-Uribe and Wang (\cite{CW}) by using different approaches. 
In \cite{ZYL}, Zhuo et al. gave the equivalent characterizations of $H^{p(\cdot)}(\mathbb R^n)$ in terms of several intrinsic square functions.
In \cite{YZN}, Yang et al. obtained the
Riesz transforms characterization for $H^{p(\cdot)}(\mathbb R^n)$. Zhuo et al. (\cite{ZSY}) also considered
the variable Hardy spaces on $RD-$spaces.
The atomic decomposition characterizations of $H^{p(\cdot)}$ is very useful when we consider the boundedness of operators on these spaces.
The atomic decomposition of Hardy spaces with variable exponents $H^{p(\cdot)}$
was established independently in \cite{CW,NS}
by using maximal function characterizations. 
Later, Sawano (\cite{S}) refined the atomic decomposition characterzations of 
$H^{p(\cdot)}(\mathbb R^n)$, and obtain some applications to the boundedness of several operators on $H^{p(\cdot)}(\mathbb R^n)$. 
The author revisited the atomic decomposition of
$H^{p(\cdot)}(\mathbb R^n)$ via Littlewood-Paley-Stein
analysis and give some applications to (sub)linear and (sub)multilinear operators in \cite{T18, Ta20, Tan20}.
Ho (\cite{Ho}) established the atomic decompositions for weighted variable Hardy spaces $H_w^{p(\cdot)}(\mathbb R^n)$. 
Moreover, the atomic decomposition characterizations
for $h^{p(\cdot)}(\mathbb R^n)$ have been studied in
\cite{Tan19, TZ}. 
Motivated by these results, we will focus on completing the real-variable theory of variable local Hardy spaces $h^{p(\cdot)}(\mathbb R^n)$, which includes that of the classical
localized Hardy space theory of Goldberg (\cite{Go}).

The main purpose of this paper is threefold. The first goal is to establish
some real-variable characterizations, including the atomic, the local vertical and the local
non-tangential maximal functions, of local variable Hardy spaces  $h^{p(\cdot)}(\mathbb R^n)$ .
The second goal is to introduce the local variable Campanato space $bmo^{p(\cdot)}(\mathbb R^n)$ and establish the duality between $h^{p(\cdot)}(\mathbb R^n)$ and $bmo^{p(\cdot)}(\mathbb R^n)$.
The third goal is to show that the boundedness of inhomogeneous Calder\'on-Zygmund singular integrals and local fractional integrals on $h^{p(\cdot)}(\mathbb R^n)$ and $bmo^{p(\cdot)}(\mathbb R^n)$.
The novelty of this paper can be summarized as follows: 
First, the approach of establishing atomic decompositions used in our paper is different from the constant exponent analogy. Indeed, we give a direct proof for the infinite atomic and finite atomic decomposition characterizations of $h^{p(\cdot)}(\mathbb R^n)$,
by avoiding the atomic decompositions of $H^{p(\cdot)}(\mathbb R^n)$. Moreover, we do not require $p^+\le1$ when we give the real-variable characterizations of $h^{p(\cdot)}(\mathbb R^n)$
and obtain the boundedness of operators on $h^{p(\cdot)}(\mathbb R^n)$. Particularly,
under certain conditions, if $f\in h^{p(\cdot)}(\mathbb R^n)$,
then there exists a sequence of special local  
$(p(\cdot),q)-$atom $\{a_j\}_j$ with supports
$\{Q_j\}_j$ and non-negative numbers $\{\lambda_j\}_j$ such that
$$
f=\sum_j\lambda_ja_j
$$
and
$$
\|f\|_{h^{p(\cdot)}(\mathbb R^n)}\sim
\left\|\sum_j\lambda_j\chi_{Q_j}\right\|_{L^{p(\cdot)}(\mathbb R^n)}.
$$
Finally yet importantly, we develop a complete dual
spaces theory of $h^{p(\cdot)}(\mathbb R^n)$
for $0<p^-\le p^+<\infty$. By products, we also establish the dual space for $H^{p(\cdot)}(\mathbb R^n)$ of $p^+>1$ and $p^-\le 1$, which gives a complete answer to the open question proposed by Izuki et.al in \cite{INS}.

The remainder of this paper is organized as follows. In Section 2, we
recall some precise definitions concerning variable Lebesgue spaces
and state the necessary lemmas which is useful in the subsequent sections.
In Section 3, we first recall the $h^{p(\cdot)}(\mathbb R^n)$ via the Littlewood-Paley-Stein theory.
Next, we give the equivalent characterizations via
the local vertical and non-tangential maximal function.
Then a new finite atomic decomposition for the local variable Hardy spaces is established
in Section 4. In section 5, we introduce a local variable Campanato space $bmo^{p(\cdot)}(\mathbb R^n)$ which is further proved to be the dual space of $h^{p(\cdot)}(\mathbb R^n)$.
Finally, in Section 6, we show that
inhomogeneous Calder\'on-Zygmund singular integrals and local
fractional integrals are bounded on $h^{p(\cdot)}(\mathbb R^n)$ and their duals.

Throughout this paper, $C$ or $c$ denotes a positive constant that may vary at each occurrence
but is independent to the main parameter, and $A\sim B$ means that there are constants
$C_1>0$ and $C_2>0$ independent of the the main parameter such that $C_1B\leq A\leq C_2B$.
Given a measurable set $S\subset \mathbb{R}^n$, $|S|$ denotes the Lebesgue measure and $\chi_S$
means the characteristic function.
Let $\mathcal S$ be the space of Schwartz functions and let $\mathcal S'$
denote the space of tempered distributions.
We also use the notations $j\wedge j'=\min\{j,j'\}$ and $j\vee j'=\max\{j,j'\}$.
Let $p(\cdot):\mathbb R^n\rightarrow (0,\infty]$
be a Lebesgue measurable function. We write $\mathbb N=\{0,1,2,\cdots\}$.
For a measurable subset $E\subset \mathbb{R}^n$, we denote
$p^-(E)= \inf_{x\in E}p(x)$ and $p^+(E)= \sup_{x\in E}p(x).$
Especially, we denote $p^-=p^{-}(\mathbb{R}^n)$ and $p^+=p^{+}(\mathbb{R}^n)$. We also write $p_-=p^-\wedge 1$.
Let $p(\cdot)$: $\mathbb{R}^n\rightarrow(0,\infty)$ be a measurable
function with $0<p^-\leq p^+ <\infty$ and $\mathcal{P}^0$
be the set of all these $p(\cdot)$.
Let $\mathcal{P}$ be the set of all measurable functions
$p(\cdot):\mathbb{R}^n \rightarrow[1,\infty) $ such that
$1<p^-\leq p^+ <\infty.$

\section{Preliminaries}

In this section, we will recall some definitions and basic results on the variable Lebesgue spaces.
For brevity, we write $X(\mathbb R^n)=X$, where
$X$ is some function space.

\begin{defn}\label{s1de1}~(\cite{CF,DHHR})\quad
The variable Lebesgue space $L^{p(\cdot)}$ is defined as the set of all
measurable functions $f$ for which the quantity
$\int_{\mathbb{R}^n}|\varepsilon f(x)|^{p(x)}dx$ is finite for some
$\varepsilon>0$ and
$$\|f\|_{L^{p(\cdot)}}=\inf{\left\{\lambda>0: \int_{\mathbb{R}^n}\left(\frac{|f(x)|}{\lambda}\right)^{p(x)}dx\leq 1 \right\}}.$$
\end{defn}

As a special case of the theory of Nakano and Luxemberg, we see that $L^{p(\cdot)}$
is a quasi-normed space. Especially, when $p^-\geq1$, $L^{p(\cdot)}$ is a Banach space.
Note that the variable exponent spaces, such as the variable Lebesgue spaces and the variable Sobolev spaces,
were studied by a substantial number of researchers (see, for instance, \cite{CFMP,KR}).
In the study of variable exponent function spaces it is common
to assume that the exponent function $p(\cdot)$ satisfies the $LH$
conditions.
We say that $p(\cdot)\in LH$, if $p(\cdot)$ satisfies

 $$|p(x)-p(y)|\leq \frac{C}{-\log(|x-y|)} ,\quad |x-y| \leq 1/2$$
and
 $$|p(x)-p(y)|\leq \frac{C}{\log|x|+e} ,\quad |y|\geq |x|.$$

Let $\mathcal{B}$ be the set of $p(\cdot)\in \mathcal{P}$ such that the
Hardy-littlewood maximal operator $M$ is bounded on  $L^{p(\cdot)}$.
It is well known
that $p(\cdot)\in \mathcal{B}$ if $p(\cdot)\in \mathcal{P}\cap LH.$
Moreover, examples shows that the above $LH$ conditions are necessary in certain sense, see Pick and R\.{u}\u{z}i\u{c}ka (\cite{PR}) for more details.
We also need the following boundedness of the vector-valued maximal operator $M$,
whose proof can be found in \cite{CFMP}.
\begin{lem}\label{s2l1}\quad Let $p(\cdot)\in \mathcal P^0\cap LH$.
Then
for any $q>1$, $f=\{f_i\}_{i\in \mathbb{Z}}$, $f_i\in L_{loc}$, $i\in \mathbb{Z}$
\begin{equation*}
  \|\|\mathbb{M}(f)\|_{l^q}\|_{L^{p(\cdot)}}\leq C\|\|f||_{l^q}\|_{L^{p(\cdot)}},
\end{equation*}
where $\mathbb{M}(f)=\{M(f_i)\}_{i\in\mathbb{Z}}$.
\end{lem}

Given a measurable function $w>0$, for $1<p<\infty$,
it is said that $w\in A_p$ if
\begin{equation*}
  [w]_{A_p}=\sup_{B}\left(\frac{1}{|B|}\int_{B}w(x)dx\right)
  \left(\frac{1}{|B|}\int_{B}w(x)^{1-p'}dx\right)^{p-1}<\infty,
\end{equation*}
where the supremum is taken over all balls $B\subset \mathbb{R}^n$.
Define the set
$$A_\infty=\cup_{p\ge1}A_p.$$
The extrapolation in variable Lebesgue space is very useful when the corresponding weighted norm inequalities is known.
\begin{lem} \label{s2l2}(\cite{CFMP}) $\mathcal{F}$ denote a family of ordered pairs of
non-negative measurable functions $(f, g)$. Suppose that their exist some $p_0$ with $0<p_0<\infty$
and every weight $w_0\in A_{\infty}$ such that
\begin{equation*}
  \int_{\mathbb{R}^n}|f(x)|^{p_0}w_0dx\leq \int_{\mathbb{R}^n}|g(x)|^{p_0}w_0dx, \quad (f,g)\in \mathcal{F},
\end{equation*}
for 
$f \in L^{p}_{w_0}$. If $p(\cdot) \in LH\bigcap\mathcal{P}^0$, then for any $(f, g)\in \mathcal{F}$ and
$f \in L^{p(\cdot)}$, we have
$$\|f\|_{L^{p(\cdot)}}\leq C\|g\|_{L^{p(\cdot)}}.$$
\end{lem}

The following generalized H\"{o}lder inequality on variable Lebesgue spaces
have been proved in \cite{CF}.

\begin{lem}\label{s2l3}
\quad Given the exponent function $p_i(\cdot)\in \mathcal{P}^0,$ define
$p(\cdot)\in \mathcal{P}^0$ by
$$\frac{1}{p(x)}=\sum_{i=1}^2\frac{1}{p_i(x)},$$
where $i=1,2.$
Then for all $f_i\in L^{p_i(\cdot)}$ and
$f_1f_2\in L^{p(\cdot)}$ and
$$\|\prod_{i=1}^2f_i\|_{p(\cdot)}\leq C\prod_{i=1}^2\|f_i\|_{p_i(\cdot)}.$$
\end{lem}

\begin{lem}(\label{s2l4}\cite{CW})
\quad Given an exponent function $p(\cdot)\in \mathcal{P}^0$
with $p^-\le 1$, 
then for all $f,\;g\in L^{p(\cdot)}$,
$$\|f+g\|^{p^-}_{L^p(\cdot)}\leq \|f\|^{p^-}_{L^{p(\cdot)}}+\|g\|^{p^-}_{L^{p(\cdot)}}.$$
\end{lem}

\begin{lem}(\label{s2l5}\cite{CF})
\quad Given an exponent function $p(\cdot)\in \mathcal{P},$ if $E\subset\mathbb R^n$ is such that $|E|<\infty,$
then $\chi_E\in L^{p(\cdot)}$ and
$$\|\chi_E\|_{L^p(\cdot)}\leq |E|+1.$$
\end{lem}

The following key lemma also plays a key role in the proofs of the main results.
The Grafakos-Kalton lemma was first established
in \cite{GK} when they consider the multlinear Calder\'on-Zygmund operators on Hardy spaces
(also see \cite{CMN1} on the weighted extension).
We need the variable exponent version as follows.

\begin{lem}(\cite{CMN})\label{s2l6}
Let $q(\cdot)\in LH\cap \mathcal P^0.$
Suppose that we are given a sequence of cubes
$\{Q_j\}_{j=1}^\infty$ and a sequence of non-negative
functions $\{F_j\}_{j=1}^\infty$.
Then for any q such that $(1\vee p^+)<q<\infty$ we have
\begin{align*}
\left\|\sum_{j=1}^\infty\chi_{Q_j}F_j\right\|_{L^{q(\cdot)}}
\le C\left\|\sum_{j=1}^\infty\left(\frac{1}{|Q_j|}
\int_{Q_j}F_j^q(y)dy\right)^{\frac{1}{q}}\chi_{Q_j}\right\|_{L^{q(\cdot)}}.
\end{align*}
\end{lem}

\begin{lem} (\cite{DHHR})\label{s2l7}
 Suppose that $p(\cdot)\in LH$
and $0<p^-\leq p^+<\infty$.

(1)\quad For all cubes (or balls) $|Q|\leq 2^n$ and any $x\in Q$, we have
$$\|\chi_Q\|_{{p(\cdot)}}\sim |Q|^{1/p(x)}.$$

(2)\quad For all cubes (or balls) $|Q|\geq1$, we
have
$$\|\chi_Q\|_{{p(\cdot)}}\sim |Q|^{1/p_\infty}, $$
where $ p_{\infty}=\lim_{x\rightarrow\infty}p(x).$
\end{lem}

\section{Local variable Hardy spaces and their maximal characterizations}

In this section, we will give several equivalent characterizations for local variable Hardy space. To state the results, we need some definitions. 
We write $\psi_t(x)=t^{-n}\psi(t^{-1}x)$ for all $x\in\mathbb R^n$.
Let $\mathcal D$ denote the set of all $C^\infty$ functions on $\mathbb R^n$
with compact supports, equipped with the inductive limit topology, and
$\mathcal D'$ its topological dual space, equipped with the weak-$*$ topology.
For $N\in\mathbb N_0$, $|\alpha|\le N+1$ and $R\in(0,\infty)$, let
$$
\mathcal D_{N,R}=\{\psi\in\mathcal D:\mbox{supp}(\psi)\subset B(0,R),
\int\psi\neq0,\|D^\alpha\psi\|_\infty\le1\}.
$$

First, we recall the local vertical, non-tangential grand maximal functions as follows.
\begin{defn}\label{s3d1}~\quad
For any $f\in\mathcal D'$, the local vertical grand maximal function 
$\mathcal G_{N,R}(f)$ of $f$ is defined by setting, for all $x\in\mathbb R^n$,
$$
\mathcal G_{N,R}(f)(x)\equiv\sup_{t\in(0,1)}\{|\psi_t\ast f(x)|:\psi\in
\mathcal D_{N,R}\},
$$
and the local non-tangential grand maximal function 
$\tilde{\mathcal G}_{N,R}(f)$ of $f$ is defined by setting, 
for all $x\in\mathbb R^n$,
$$
\tilde{\mathcal G}_{N,R}(f)(x)\equiv\sup_{|x-z|<t<1}\{|\psi_t\ast f(z)|:\psi\in
\mathcal D_{N,R}\}.
$$
For convenience, we write $\mathcal G_{N,1}(f)=\mathcal G_{N}^0(f)$
and $\tilde{\mathcal G}_{N,1}(f)=\tilde{\mathcal G}_{N}^0(f)$
and also write $\mathcal G_{N,2^{3(10+n)}}(f)=\mathcal G_{N}(f)$
and $\tilde{\mathcal G}_{N,2^{3(10+n)}}(f)=\tilde{\mathcal G}_{N}(f)$.
Obviously, $$
\mathcal G_{N}^0(f)\le \mathcal G_{N}(f)\le\tilde{\mathcal G}_{N}(f).
$$
\end{defn}

Next, we also recall the local vertical, tangential and non-tangential
maximal functions.

\begin{defn}\label{s3d2}~\quad
Let $
\psi_0\in\mathcal D$ with $\int \psi_0(x)dx\neq 0$.
The local vertical maximal
function $\mathcal M_{\psi_0}(f)$ is defined by
$$
\mathcal M_{\psi_0}(f)(x)\equiv \sup_{j\in\mathbb Z_+}(\psi_0)_j\ast f(x)|.
$$
For $j\in\mathbb Z_+$, $A,\,B\ge1$, the local tangential peetre-type maximal function $\psi^{\ast,\ast}_{0,A,B}(f)$
is defined by
$$
\psi^{\ast,\ast}_{0,A,B}(f)(x)\equiv \sup_{j\in\mathbb Z_+,y\in\mathbb R^n}
\frac{(\psi_0)_j\ast f(x-y)|}{(1+2^j|y|)^A2^{B|y|}},
$$
The local non-tangantial maximal
function $\mathcal M^\ast_{\psi_0}(f)$ is defined by
$$
\mathcal M^\ast_{\psi_0}(f)(x)\equiv \sup_{|x-y|<t<1}(\psi_0)_t\ast f(x)|.
$$
Hereafter, $(\psi_0)_j(x)=2^{jn}\psi_0(2^jx)$ and $(\psi_0)_t(x)=(1/{t^n})\psi_0(x/t)$.
\end{defn}

We also introduce the local Littlewood-Paley-Stein square function below.
\begin{defn}\label{s3d3}~\quad
Let $\varphi\in \mathcal{D}(\mathbb{R}^n)$ satisfy
\begin{equation*}
  \mbox{supp}\;\widehat{\varphi}(\xi)\subset \{\xi:1/2<|\xi|\leq2\},
\end{equation*}

\noindent and $\Phi$ whose Fourier transform does not vanish at the origin with
\begin{equation*}
  \mbox{supp}\;\widehat{\Phi}\subset\{\xi:|\xi|\leq 2\}
\end{equation*}
\noindent satisfy
\begin{equation*}
  |\widehat{\Phi}(\xi)|^2
  +\sum_{j=1}^\infty|\widehat{\varphi}(2^{-j}\xi)|^2=1,\quad \mbox{for}\;
  \mbox{all}\; \xi\in \mathbb{R}^n.
\end{equation*}

We denote $\Phi=\varphi_0$.
For $f\in \mathcal D'$,
we give the definition of local Littlewood-Paley-Stein square function
\begin{align*}
\mathcal{G}_{loc}(f)(x):=\left(\sum_{j\in \mathbb N}|\varphi_j\ast f(x)|^2\right)^{1/2},
\end{align*}

and the discrete Littlewood-Paley-Stein square function
\begin{align*}
\mathcal{G}_{loc}^d(f)(x):=\left(\sum_{j\in \mathbb N}
\sum_{\mathbf k\in\mathbb Z^n}|\varphi_j\ast f(2^{-j}\mathbf k)|^2\chi_Q(x)\right)^{1/2},
\end{align*}
where $Q$ denote dyadic cubes in $\mathbb R^n$ with side-lengths $2^{-j}$ and the
lower left-corners of $Q$ are $2^{-j}\mathbf k$.
\end{defn}

We now recall the local variable Hardy spaces via  local Littlewood-Paley-Stein square function in \cite{Tan19}.
\begin{defn}\label{s3d4}~\quad
Let $f\in \mathcal{D'}$, $p(\cdot)\in {\mathcal{P}^0}$.
The localized Hardy space with variable exponent ${h}^{p(\cdot)}$
is the set of all $f\in \mathcal{D}^\prime$ for which the quantity
$$\|f\|_{{h}^{p(\cdot)}}=\|\Phi\ast f\|_{{L}^{p(\cdot)}}
+\Bigg\|\Big\{\sum_{j=1}^\infty|\varphi_j\ast f|^2\Big\}^{1/2}\Bigg\|_{{L}^{p(\cdot)}}<\infty.$$
\end{defn}

\begin{rem}\label{s1r1}
Let all notation be as in Definition \ref{s3d4}, observe that
$$
\|f\|_{{h}^{p(\cdot)}} \sim
\Bigg\|\Big\{\sum_{j=0}^\infty|\varphi_j\ast f|^2\Big\}^{1/2}\Bigg\|_{{L}^{p(\cdot)}}.
$$
Also, it is shown in \cite[Theorem 1.3]{Tan19} that $\|f\|_{{h}^{p(\cdot)}}
\sim\|\mathcal{G}_{loc}^d(f)\|_{{L}^{p(\cdot)}}$
\end{rem}

Now, let us state the main results in this section.
We first obtain the equivalent characterizations of $h^{p(\cdot)}$ as follows.

\begin{thm}\label{s3th1}\quad
Let $f\in \mathcal{D'}$, $p(\cdot)\in \mathcal P^0\cap LH$ .
For a fixed large integer $N_0$, any $N\ge N_0$, the following statements
are mutually equivalent:\\$
(1)\; f\in h^{p(\cdot)};\\
(2)\; \mathcal G_N(f)\in L^{p(\cdot)};\\
(3)\; \mathcal G^0_N(f)\in L^{p(\cdot)};\\
(4)\; \tilde{\mathcal G}_N(f)\in L^{p(\cdot)};\\
(5)\; \tilde{\mathcal G}^0_N(f)\in L^{p(\cdot)};\\
(6)\; \mathcal M_{\psi_0}(f)\in L^{p(\cdot)};\\
(7)\; \mathcal M^\ast_{\psi_0}(f)\in L^{p(\cdot)};\\
(8)\; \psi^{\ast,\ast}_{0,A,B}(f)\in L^{p(\cdot)}.\\
$
Moreover,
for all $f\in \mathcal {D'}$,
\begin{align*}
\|f\|_{h^{p(\cdot)}}&\sim \|\mathcal G_N(f)\|_{{L}^{p(\cdot)}} \sim \|\mathcal G^0_N(f)\|_{L^{p(\cdot)}}\sim \|\tilde{\mathcal G}_N(f)\|_{L^{p(\cdot)}}\\
&\sim \|\tilde{\mathcal G}^0_N(f)\|_{L^{p(\cdot)}}
\sim \|\mathcal M_{\psi_0}(f)\|_{L^{p(\cdot)}}\sim \|\mathcal M^\ast_{\psi_0}(f)\|_{L^{p(\cdot)}}\sim\|\psi^{\ast,\ast}_{0,A,B}(f)\|_{L^{p(\cdot)}},
\end{align*}
where the implicit equivalent positive constants are independent of $f$.
\end{thm}

\begin{proof}
The proof of equivalence for the first two norms can be found in \cite[Section 9]{NS}
and \cite[Theorem 1.3]{Tan19}. We only need to prove that the rest of norms 
are equivalent. To end it, 
from \cite[Theorem 3.14, Corollary 3.15]{YY}, we know that
\begin{align*}
 \|\mathcal G_N(f)\|_{{L}^{p}_w} &\sim \|\mathcal G^0_N(f)\|_{{L}^{p}_w}\sim \|\tilde{\mathcal G}_N(f)\|_{{L}^{p}_w}\\
&\sim \|\tilde{\mathcal G}^0_N(f)\|_{{L}^{p}_w}\sim \|\mathcal M_{\psi_0}(f)\|_{{L}^{p}_w}\sim \|\mathcal M^\ast_{\psi_0}(f)\|_{{L}^{p}_w}\sim\|\psi^{\ast,\ast}_{0,A,B}(f)\|_{{L}^{p}_w},
\end{align*}
for every $w\in A_\infty$.
Notice that $p(\cdot) \in \mathcal{P}^0\cap LH$, by Lemma \ref{s2l2},
for $(\mathcal G_N(f)\chi_{B(0,R)}, \mathcal G^0_N(f))\in \mathcal{F}$ and
$\mathcal G_N(f)\chi_{B(0,R)}\in L^{p(\cdot)}$ with $0<R<\infty$, we have
$$\|\mathcal G_N(f)\chi_{B(0,R)}\|_{L^{p(\cdot)}}\leq C\|\mathcal G^0_N(f)\|_{L^{p(\cdot)}}.$$
If we take the limit as $R\rightarrow\infty$, then by Fatou’s lemma
$$\|\mathcal G_N(f)\|_{L^{p(\cdot)}}\leq C\|\mathcal G^0_N(f)\|_{L^{p(\cdot)}}.$$
Similarly, for $(\mathcal G^0_N(f)\chi_{B(0,R)}, \mathcal G_N(f))\in \mathcal{F}$ and
$\mathcal G^0_N(f)\in L^{p(\cdot)}$, we have
$$\|\mathcal G^0_N(f)\|_{L^{p(\cdot)}}\leq C\|\mathcal G_N(f)\|_{L^{p(\cdot)}}.$$
Repeating the same argument, we can get the desired result.
\end{proof}

\begin{rem}\label{s3r2}
If $f\in \mathcal{D'}$, $p(\cdot)\in \mathcal{P}\cap LH$, we want to stress that
the function spaces ${h}^{p(\cdot)}$ and ${L}^{p(\cdot)}$ are isomorphic to each
other. To see this, we only need to observe that in \cite
[Proposition 2.2]{Tang}, if $f\in A_p$ with $p\in(1,\infty),$
then $f\in L_w^p$ if and only if $f\in\mathcal D'$ and $\mathcal G_N^0(f)
\in L_w^p$ with $\|f\|_{L^p_w}\sim \|\mathcal G_N^0(f)\|_{L^p_w}$.
Then applying the \cite[Corollary 1.11]{CFMP}, we get that
$\|f\|_{L^{p(\cdot)}}\sim \|\mathcal G_N^0(f)\|_{L^{p(\cdot)}}$.
We also remark that the $h^{p(\cdot)}-$norm is stronger than the
topology of $\mathcal D'$; indeed, for any $\psi\in\mathcal D$
and $\mbox{supp}\;\psi\subset B_0=B(0,1)$, by Lemma
\ref{s2l3} and Lemma \ref{s2l5},
we have 
\begin{align*}
|\big<f,\psi\big>|^{p_-}&=|f\ast \tilde\psi(0)|^{p_-}
\le C\inf_{y\in B_0}\mathcal M_N^0(f)(y)^{p_-}
\le C\frac{1}{|B_0|}\int_{B_0}\mathcal M_N^0(f)(y)^{p_-}dy
\\
&\le C\|\mathcal M_N^0(f)\|_{L^p(\cdot)}^{p_-}\|\chi_{B_0}\|_{L^{(p(\cdot)/{p_-})'}}\le C\|f\|^{p_-}_{h^{p(\cdot)}},
\end{align*}
where $\tilde \psi(x)=\psi(-x)$.
We also remark that, as a consequence of \cite[Theorem 1.3]{Tan19}, 
$L^q\cap h^{p(\cdot)}$ is dense in $h^{p(\cdot)}$
for $1\le q<\infty$.

\end{rem}

We also obtain the completeness of $h^{p(\cdot)}$ that
are of interest in their own right. We have proved it implicitly in \cite{Tan19} by applying the Littlewood-Paley-Stein theory. Here we give a different approach.
\begin{prop}\label{s3p1}
Given $p(\cdot)\in \mathcal{P}^0\cap LH$, the space $h^{p(\cdot)}$ is 
complete with respect to the norm $\|\cdot\|_{h^{p(\cdot)}}$.
\end{prop}
\begin{proof}
For any $\psi\in\mathcal D_{N,0}$, by the remark, if any sequence $\{f_k\}$
converges in $h^{p(\cdot)}$, then it also converges in
$\mathcal D'$.
We only need to consider the case $p^-\le 1$, the other case is similar but easier.  Given any sequence $\{f_k\}$ in $h^{p(\cdot)}$ fulfilling that
$$
\sum_k\|f_k\|_{h^{p(\cdot)}}^{p^-}<\infty.
$$
We denote that $F_j=\sum_{k=1}^j f_k$.
Then by Lemma \ref{s2l4}, we obtain that
the sequence $\{F_j\}$ is Cauchy in $h^{p(\cdot)}$ and so in $\mathcal D'$
and that
$$
\|f\|_{h^{p(\cdot)}}^{p^-}=\|\sum_kf_k\|_{h^{p(\cdot)}}^{p^-}
\le \sum_k\|f_k\|_{h^{p(\cdot)}}^{p^-}<\infty.
$$
Therefore,
$$
\|f-F_j\|_{h^{p(\cdot)}}^{p^-}\le 
\sum_{k\ge j+1}\|f_k\|_{h^{p(\cdot)}}^{p^-}\rightarrow 0.
$$
as $j$ tends to $\infty$. So the series converges to $f$ in $h^{p(\cdot)}$.
The proof is complete.
\end{proof}

 \section{Atomic characterizations of $h^{p(\cdot)}$}
 In this section, we will give new atomic decompositions of $h^{p(\cdot)}$ using special local $(p(\cdot),q)-$atom. For one thing, we will discuss the infinite atomic characterizations of $h^{p(\cdot)}$. For another, we will
 obtain the finite atomic decompositions of $h^{p(\cdot)}$.
 \subsection{Infinite local $(p(\cdot),q)-$atomic decompositions} 
The first aim of this chapter is to sharpen the atomic decomposition theory of
$h^{p(\cdot)}$, which have been established by the author \cite{Tan19}. In what follows, we introduce the new definitions for the special local $(p(\cdot),q)$-atom of $h^{p(\cdot)}$.
Denote $Q(x,\ell(Q))$ the closed cube centered at $x$ and of sidelength $\ell(Q)$. Similarly, given $Q=Q(x,\ell(Q))$ and $\lambda>0$, $\lambda Q$ means that the cube with the same center $x$ and with sidelength $\lambda\ell(Q)$. We denote $\tilde Q=2\sqrt{n}Q$ simply .  

\begin{defn}\label{s4d1}~\quad
Let $p(\cdot): \mathbb{R}^n\rightarrow (0,\infty)$,
$p(\cdot)\in \mathcal P^0$ and $1<q\leq \infty$.
Fix an integer $d\geq d_{p(\cdot)}\equiv \min\{d\in \mathbb{N}: p^-(n+d+1)>n\}$.
A function $a$ is said to be a special local $(p(\cdot),q)$-atom of $h^{p(\cdot)}$ if\\
(i)\;$supp\;a\subset Q$;\\
(ii)\;$\|a\|_{L^q}\leq |Q|^{1/q}$;\\
(iii)\;$\int_{\mathbb{R}^n} a(x)x^\alpha dx=0$ for all $|\alpha| \leq d$, if $|Q|<1$.
\end{defn}

In Definition \ref{s4d1}, if the condition $(iii)$ is replaced by the condition 
$(iii)'$: $\int_{\mathbb{R}^n} a(x)x^\alpha dx=0$ for all $|\alpha| \leq d$ and all cubes $Q\subset\mathbb R^n$, then the function $a$ is said to be a special
$(p(\cdot),q)$-atom of $H^{p(\cdot)}$.
Observe that for any $1<q<\infty$, all
$(p(\cdot),\infty)-$atoms are  $(p(\cdot),q)-$atoms, since that $|Q|^{-1/q}\|a\|_{L^q}\le\|a\|_{L^\infty}$.
The following theorems improves the previous atomic decomposition results in \cite{Tan19}.
\begin{thm}\label{s4th1}\quad Let $p(\cdot)\in \mathcal P^0\cap LH$. Suppose that $p^+<q\le \infty$ when $p^+\ge 1$ and $1<q\le\infty$ when
$p^+<1$. Fix an integer $d\geq d_{p(\cdot)}\equiv \min\{d\in \mathbb{N}: p^-(n+d+1)>n\}$.
Given countable collections of cubes $\{Q_j\}_{j=1}^\infty$,
of non-negative coefficients  $\{\lambda_j\}_{j=1}^\infty$ and of the special local $(p(\cdot),q)$-atoms $\{a_j\}_{j=1}^\infty$.
If 
\begin{equation*}
\left\|\sum_{j=1}^\infty\lambda_j\chi_{Q_j}\right\|_{L^{p(\cdot)}}
\end{equation*}
is finite. Then the series $f=\sum_{j=1}^\infty \lambda_ja_j$
converges in $h^{p(\cdot)}$ and satisfies
\begin{equation*}
\|f\|_{h^{p(\cdot)}}\le C\left\|\sum_{j=1}^\infty\lambda_j\chi_{Q_j}\right\|_{L^{p(\cdot)}}.
\end{equation*}
\end{thm}

\begin{thm}\label{s4th2}\quad Let $p(\cdot)\in \mathcal P^0\cap LH$, $1<q\le\infty$ and $s\in (0,\infty)$. 
If $f\in h^{p(\cdot)}$, then there exists non-negative coefficients  $\{\lambda_j\}_{j=1}^\infty$ and the 
special local $(p(\cdot),q)$-atoms $\{a_j\}_{j=1}^\infty$ 
such that
$f=\sum_{j=1}^\infty \lambda_ja_j$,
where the series converges almost everywhere and in $\mathcal D'$, and that
\begin{equation*}
\left\|\left(\sum_{j=1}^\infty(\lambda_j\chi_{Q_j})^{s}\right)^{\frac{1}{s}}\right\|_{L^{p(\cdot)}}
\le C\|f\|_{h^{p(\cdot)}}.
\end{equation*}
\end{thm}

As an immediate corollary, we will get the following atom decomposition for $h^{p(\cdot)}$.
 
\begin{cor}\label{s4c1}\quad Let $p(\cdot)\in \mathcal P^0\cap LH$. Suppose that $p^+<q\le \infty$ when $p^+\ge 1$ and $1<q\le\infty$ when
$p^+<1$.
Then $f\in \mathcal D'$ is in $h^{p(\cdot)}$
if and only if there exists non-negative coefficients  $\{\lambda_j\}_{j=1}^\infty$ and the special local $(p(\cdot),q)$-atoms $\{a_j\}_{j=1}^\infty$ 
such that
$f=\sum_{j=1}^\infty \lambda_ja_j$,
where the series converges in $h^{p(\cdot)}$, and that
\begin{align*}
 \|f\|_{h^{p(\cdot)}}&\sim
 \inf\left\{\left\|\sum_{j=1}^\infty\lambda_j\chi_{Q_j}\right\|_{L^{p(\cdot)}}:\;
f=\sum_{j=1}^\infty \lambda_ja_j \right\}\\
&\sim
 \inf\left\{\left\|\left(\sum_{j=1}^\infty(\lambda_j\chi_{Q_j})^{p_-}\right)^{\frac{1}{p_-}}\right\|_{L^{p(\cdot)}}:\;
f=\sum_{j=1}^\infty \lambda_ja_j \right\},
\end{align*}
where the infimum is taken over all expressions as above.
\end{cor}

Now we are ready to prove Theorem \ref{s4th1}.\\
\noindent\textit{Proof of Theorem \ref{s4th1}.}\quad
Fix $
\psi\in\mathcal D_{N,0}$ with $\int \psi(x)dx\neq 0$.
By Theorem \ref{s3th1}, we know that
$$
\|f\|_{h^{p(\cdot)}}\sim\|\mathcal G_{N}^0(f)\|_{L^{p(\cdot)}}\equiv \left\|\sup_{0<t<1}|\psi_t\ast f|\right\|_{L^{p(\cdot)}}.
$$
Suppose that $\{\lambda_j\}_{j=1}^\infty$ has only a finite number of non-zero entries.
We consider the two cases for $Q$ as follows.
Case 1: $|Q|<1$. In this case, we claim that
$$
\mathcal G_{N}^0\left[\sum_{j=1}^\infty
\lambda_ja_j\right](x)\le C\sum_{j=1}^\infty
\lambda_j\left(M(a_j)(x)\chi_{\tilde Q}(x)
+M(\chi_{Q_j})(x)^{\frac{n+d+1}{n}}\right).
$$
To prove it, we only need to observe that
for any special $(p(\cdot),q)-$atom $a_j$ with $Q_j=Q(x_j,\ell(Q_j))$ and for all $x\in {\tilde {Q_j}}^c$,
\begin{equation*}
\mathcal G_{N}^0(a_j)(x)\le C\frac{\ell(Q_j)^{n+d+1}}
{(\ell(Q_j)+|x-x_j|)^{n+d+1}}.
\end{equation*}
In fact, for any $\psi\in\mathcal D_{N,0}$ and $0<t<1$,
let $P$ be the Taylor expansion of $\psi$ at the point
$\frac{(x-x_j)}{t}$ with degree $d$. By the Taylor remainder theorem, we have
\begin{equation*}
\left|\psi(\frac{x-y}{t})-P(\frac{x-x_j}{t})\right|
\le C \sum_{|\alpha|=d+1}
\Bigg|(D^\alpha \psi)
\Big(\frac{\theta(x-y)+(1-\theta)(x-x_j)}{t}\Big)\Bigg|
\Bigg|\frac{x_j-y}{t}\Bigg|^{d+1},
\end{equation*}
where multi-index $\alpha\in\mathbb Z_+^n$
and $\theta\in(0,1)$.
Since $0<t<1$ and $x\in \tilde {Q_j}^c$,
then we notice that $\mbox{supp}(a_j\ast\psi_t)
\subset B(x_j,2\sqrt{n})$ and that
$a_j\ast \psi_t(x)\neq 0$ implies that
$t>\frac{|x-x_j|}{2}$. Thus, for all $x\in \tilde {Q_j}^c$,
we have
\begin{eqnarray*}
|(a_j\ast \psi_t)(x)|
&=&\left|
t^{-n}\int_{\mathbb R^n} a_j(y)
\left(\psi(\frac{x-y}{t})-P(\frac{x-x_j}{t})\right)dy\right|\\
&\le&
C\chi_{\tilde Q_j}(x)|x-x_j|^{-(n+d+1)}\int_{Q_j}|a_j(y)||y-x_j|^{d+1}dy\\
&\le&
C\chi_{\tilde Q_j}(x)|x-x_j|^{-(n+d+1)}\ell(Q_j)^{n+d+1}.
\end{eqnarray*}
Hence, we have proved the claim. 
Since that ${M}$ is bounded on $L^q$ 
when $(1\vee p^+)<q<\infty$.
Applying the H\"older inequality
yields that
\begin{align*}
\begin{split}
  &\left(\frac{1}{|Q_{j}|}\int_{Q_{j}}
  |{M}(a_{j})(x)|^qdx\right)^{\frac{1}{q}}
  \le \frac{1}{|Q_{j}|^{1/q}}\|
  {M}(a_{j})\|_{L^{q}}\\
  &\le C\frac{1}{|Q_{j}|^{1/q}}\|a_{j}\|_{L^{q}}
  \le C.
\end{split}
\end{align*}

Choose $\tau$ such that $\tau p^->1$. Then by Lemma \ref{s2l6} and Lemma \ref{s2l1}, we get that
\begin{eqnarray*}
&&\left\|\mathcal G_{N}^0\left(\sum_{j=1}^\infty
\lambda_ja_j\right)\right\|_{L^{p(\cdot)}}\\
&\le& C\left\|\sum_{j=1}^\infty
\lambda_jM(a_j)\chi_{\tilde Q_j}\right\|_{L^{p(\cdot)}}
+C\left\|\sum_{j=1}^\infty\lambda_j\left(M\chi_{Q_j}\right)^{\frac{n+d+1}{n}}\right\|_{L^{p(\cdot)}}\\
&\le& C\left\|\sum_{j=1}^\infty
|\lambda_{j}|\left(\frac{1}{|Q_{j}|}\int_{Q_{j}}
|M(a_{j})(x)|^qdx\right)^{\frac{1}{q}}
\chi_{\tilde Q_{j}}\right\|_{L^{q(\cdot)}}
+C\left\|\sum_{j=1}^\infty\lambda_j\left(M\chi_{Q_j}\right)^{\frac{n+d+1}{n}}\right\|_{L^{p(\cdot)}}\\
&\le&C\left\|\sum_{j=1}^\infty
\lambda_{j}
\chi_{\tilde Q_{j}}\right\|_{L^{p(\cdot)}}
\le C\left\|\left(\sum_{j=1}^\infty
\lambda_{j}
M^\tau(\chi_{Q_{j}})\right)^{1/\tau}\right\|^{\tau}_{L^{\tau p(\cdot)}}
\le C\left\|\sum_{j=1}^\infty
\lambda_{j}
\chi_{Q_{j}}\right\|_{L^{p(\cdot)}}.
\end{eqnarray*}

When $q=\infty$, it is easy to see that
\begin{align*}
\begin{split}
 {M}(a_{j})(x)\le C\|a_{j}\|_{L^{\infty}}
  \le C.
\end{split}
\end{align*}
for $x\in \tilde Q_j$. Similarly, we have
\begin{eqnarray*}
&&\left\|\mathcal G_{N}^0\left(\sum_{j=1}^\infty
\lambda_ja_j\right)\right\|_{L^{p(\cdot)}}\\
&\le& C\left\|\sum_{j=1}^\infty
\lambda_jM(a_j)\chi_{\tilde Q_j}\right\|_{L^{p(\cdot)}}
+C\left\|\sum_{j=1}^\infty\lambda_j\left(M(\chi_{Q_j})\right)^{\frac{n+d+1}{n}}\right\|_{L^{p(\cdot)}}\\
&\le&C\left\|\sum_{j=1}^\infty
\lambda_{j}
\chi_{\tilde Q_{j}}\right\|_{L^{p(\cdot)}}
\le C\left\|\left(\sum_{j=1}^\infty
\lambda_{j}
M^\tau(\chi_{Q_{j}})\right)^{1/\tau}\right\|^{\tau}_{L^{\tau p(\cdot)}}
\le C\left\|\sum_{j=1}^\infty
\lambda_{j}
\chi_{Q_{j}}\right\|_{L^{p(\cdot)}}.
\end{eqnarray*}

Case 2: $|Q|\ge1$. Let $\bar Q_j=Q_j(x_j,\ell(Q)+2)$.
In this case, 
since we only need to consider $t\in(0,1)$,
we observe that
$$
\mathcal G_{N}^0\left[\sum_{j=1}^\infty
\lambda_ja_j\right](x)\le C\sum_{j=1}^\infty
|\lambda_j|M(a_j)(x)\chi_{Q_j^\ast}(x).
$$
Indeed, applying the fact that $\ell(Q)\ge1$ yields $\mbox{supp}\;(a_j\ast \psi_t)(x)\subset \bar Q_j
\subset 100 Q_j$.
Thus, by using similar but easier argument when $q<\infty$
we get that
\begin{eqnarray*}
&&\left\|\mathcal G_{N}^0\left(\sum_{j=1}^\infty
\lambda_ja_j\right)\right\|_{L^{p(\cdot)}}\\
&\le& C\left\|\sum_{j=1}^\infty
\lambda_jM(a_j)\chi_{\bar Q_j}\right\|_{L^{p(\cdot)}}\\
&\le& C\left\|\sum_{j=1}^\infty
|\lambda_{j}|\left(\frac{1}{|Q_{j}|}\int_{Q_{j}}
|M(a_{j})(x)|^qdx\right)^{\frac{1}{q}}
\chi_{\bar Q_{j}}\right\|_{L^{q(\cdot)}}\\
&\le&C\left\|\sum_{j=1}^\infty
\lambda_{j}
\chi_{100 Q_{j}}\right\|_{L^{p(\cdot)}}
\le C\left\|\sum_{j=1}^\infty
\lambda_{j}
\chi_{Q_{j}}\right\|_{L^{p(\cdot)}}.
\end{eqnarray*}
When $q=\infty$, we can get the desired result similarly.

Finally, we extend the result to the general case.
Given countable collections of cubes $\{Q_j\}_{j=1}^\infty$,
of non-negative coefficients  $\{\lambda_j\}_{j=1}^\infty$ and of the special local $(p(\cdot),q)$-atoms $\{a_j\}_{j=1}^\infty$.
Observe that
\begin{equation*}
\left\|\sum_{j=1}^\infty\lambda_j\chi_{Q_j}\right\|_{L^{p(\cdot)}}
<\infty
\end{equation*}
and that for $1\le m\le n<\infty$
\begin{equation*}
\left\|\mathcal G_{N}^0\left(\sum_{j=m}^n
\lambda_ja_j\right)\right\|_{L^{p(\cdot)}}
\le C\left\|\sum_{j=m}^n
\lambda_{j}
\chi_{Q_{j}}\right\|_{L^{p(\cdot)}}.
\end{equation*}

Hence, the sequence $\{\lambda_ja_j\}_{j=1}^\infty$ is Cauchy
in $h^{p(\cdot)}$ and converges to an element $f\in h^{p(\cdot)}$. From the Remark \ref{s3r2}, we know that the $h^{p(\cdot)}-$norm is stronger than the
topology of $\mathcal D'$. So the sequence $\{\lambda_ja_j\}_{j=1}^\infty$ is also converges to $f$ in $\mathcal D'$. 
By Fatou's lemma, we obtain
\begin{equation*}
\|f\|_{h^{p(\cdot)}}
\le C
\lim_{n\rightarrow\infty}\left\|\mathcal G_{N}^0\left(\sum_{j=1}^n
\lambda_ja_j\right)\right\|_{L^{p(\cdot)}}
\le C\left\|\sum_{j=1}^\infty\lambda_j\chi_{Q_j}\right\|_{L^{p(\cdot)}}
\end{equation*}
Therefore, we have completed the proof of this theorem.
$\hfill\Box$

Before we prove the next theorem, 
we revisit the Calder\'on-Zygmund decomposition associated with the local grand maximal function on $\mathbb R^n$. For more detail,
we refer to \cite[Section 4]{Tang} (also see \cite[Section 5]{B} and \cite[p.102-105,\;p.110-111]{St}).  
Let $d\in\mathbb N\cup\{0\}$ be some fixed integers
and $\mathcal P_d$ denote the linear space pf polynomials
in $n$ variables of degrees no more than $d$.  
For each $i$ and $P\in \mathcal P_d$, set
\begin{align*}
\|P\|_i\equiv\left[\frac{1}{\int\eta_i(x)dx}\int_{\mathbb R^n}|P(x)^2\eta_i(x)dx|\right]^{1/2}.
\end{align*}
Then $(\mathcal P_d,\|\cdot\|_i)$ is finite dimensional Hilbert
space. Let $f\in\mathcal D'$. Since $f$ induces a linear functional on $\mathcal P_d$ via 
$Q\longmapsto 1/\int\eta_i(x)dx\big<f,Q\eta_i\big>$,
by the Riesz lemma, there exists a unique polynomial 
$P_i\in \mathcal P_d$ for each $i$ such that for all 
$Q\in\mathcal P_d$,
\begin{align*}
\frac{1}{\int_{\mathbb R^n}\eta_i(x)dx}
\big<f, Q\eta_i\big>=\frac{1}{\int_{\mathbb R^n}\eta_i(x)dx}
\big<P_i,Q\eta_i\big>.
\end{align*}

\begin{lem}\label{s4l1}
Let $d\in \mathbb Z$ and $\lambda>0$.
Suppose that $f\in\mathcal D'$ and 
$\Omega=\{x\in\mathbb R: \mathcal G_Nf(x)>\lambda\}$.
Fix $a=1+2^{-(11+n)}$ and $b=1+2^{-(10+n)}$.
Then there exist collections of closed
cubes $\{Q_k\}$ whose interiors distance from
$\Omega^c$ such that $\Omega=\bigcup_{k}Q_k$
and $Q_k\subset aQ_k\subset bQ_k$.
Moreover, this gives us collections of 
$\{Q_k^\ast\}$ and functions $\{\eta_k\}\subset \mathcal D$, and a decomposition $f=g+b$, $b=\sum_kb_k$,
such that\\
(a)\; $\bigcup Q_k^\ast=\Omega$ and the $\{Q_k^\ast\}$ have
the bounded interior property: every point is contained in at
most a fixed number of the $\{Q_k^\ast\}$.\\
(b)\; $\chi_{\Omega}=\sum_k\eta_k$, with $0\le\eta\le1$
and each function $\eta_k$ is supported in $Q_k^\ast$.\\
(c)\; The distribution function $g$ fulfilling the following inequality:
\begin{align*}
\mathcal G_N^0(g)(x)\le \chi_{\Omega^c}\mathcal G_N^0(f)(x)
+C\lambda\sum_i\frac{\ell_i^{n+d+1}}
{(\ell_i+|x-x_i|)^{n+d+1}}.
\end{align*}\\
(d)\; If $f\in L^1_{loc}$, then $g\in L^\infty$ with $|g(x)|\le C\lambda$
for $a.e.\;x\in\mathbb R^n$.\\
(e)\; The distribution function $b_i$ is define by 
$b_i=(f-P_i)\eta_i$ if $\ell_i<1$, otherwise we set
$b_i=f\eta_i$ with a polynomial $P_i\in\mathcal P_d$
such that $\int_{\mathbb R^n}b_i(x)q(x)dx=0$ for
all $q\in\mathcal P_d$. Then 
if $x\in Q_i^\ast$, $\mathcal G_N^0(b_i)(x)
\le C\mathcal G_N(f)(x)$
and if $x\in (Q_i^\ast)^c$,
\begin{align*}
\mathcal G_N^0(b_i)(x)\le C
\frac{\lambda\ell_i^{n+d+1}}{(\ell_i+|x-x_i|)^{n+d+1}}.
\end{align*}
Hereafter, $x_i$ and $\ell_i$ denote the center and the sidelength of $Q_i$, respectively.
\end{lem}

In what follows, we denote $E_1^k=\{i\in\mathbb N:
|Q_i^k|\ge 1/(2^4n)\}$ and$E_2^k=\{i\in\mathbb N:
|Q_i^k|\le 1/(2^4n)\}$, $F_1^k=\{i\in\mathbb N:
|Q_i^k|\ge 1)\}$ and $F_2^k=\{i\in\mathbb N:
|Q_i^k|\le 1)\}$. 
Next, we will give the proof of Theorem \ref{s4th2}.\\
\noindent\textit{Proof of Theorem \ref{s4th2}.}\quad
First we assume that $f\in h^{p(\cdot)}\cap L^2$. 
For each $j\in\mathbb Z$, consider the level set
$\Omega_k=\{x\in\mathbb R^n: \mathcal G(f)(x)>2^k\}$.
Then it follows that $\Omega_{k+1}\subset\Omega_{k}$.
By Proposition \ref{s4l1},
$f$ admits
a Calder\'on-Zygmund  decomposition of degree $d$ and 
height $2^k$
associated with $\mathcal G_N(f)$,
$$
f=g^k+\sum_i b_i^k,\quad\mbox{in}\quad \mathcal D',
$$
where $b^k_i=(f-P^k_i)\eta^k_i$ if $\ell^k_i<1$ and $b^k_i=f\eta^k_i$
if $\ell^k_i\ge 1$.
We claim that $g^k\rightarrow f$ in both $h^{p(\cdot)}$
and $\mathcal D'$ as $k\rightarrow\infty$.
Indeed, applying Lemma \ref{s4l1}
\begin{align*}
\left\|f-g^k\right\|^{p_-}_{h^{p(\cdot)}}
&\le \left\|\sum_i \mathcal G_N^0b_i^k\right\|^{p_-}_{L^{p(\cdot)}}\\
&\le \left\|\sum_i \chi_{Q_{i,k}^\ast}\mathcal G_Nf\right\|^{p_-}_{L^{p(\cdot)}}
+\left\|\sum_i \frac{2^k\ell_i^{n+d+1}\chi_{(Q_{i,k}^\ast)^c}}
{(\ell_i+|\cdot-x_i|)^{n+d+1}}\right\|^{p_-}_{L^{p(\cdot)}}\\
&\le C\left\|\sum_i \chi_{Q_{i,k}^\ast}\mathcal G_Nf\right\|^{p_-}_{L^{p(\cdot)}}
+C\left\|\sum_i (M\chi_{Q_{i,k}^\ast})^{\frac{n+d+1}{n}}\right\|^{p_-}_{L^{p(\cdot)}}\\
&\le C\Bigg\| \chi_{\Omega_k}\mathcal G_Nf\Bigg\|^{p_-}_{L^{p(\cdot)}}
+C\left\|\sum_i 2^k\chi_{Q_{i,k}^\ast}\right\|^{p_-}_{L^{p(\cdot)}}\\
&\le C\Big\| \chi_{\Omega_k}\mathcal G_Nf\Big\|^{p_-}_{L^{p(\cdot)}}
+C\Big\|2^k\chi_{\Omega_k}\Big\|^{p_-}_{L^{p(\cdot)}}\\
&\le C\left\|\chi_{\Omega_k}\mathcal G_Nf\right\|^{p_-}_{L^{p(\cdot)}}
\end{align*}
Then we obtain that 
$$
\left\|f-g^k\right\|_{h^{p(\cdot)}}
=\left\|\sum_i b_i^k\right\|_{h^{p(\cdot)}}\rightarrow 0
$$
as $i\rightarrow\infty$.
Moreover, notice that
$\|g^k\|_{L^\infty}\le C2^j$ it follows that $g^k\rightarrow 0$ 
uniformly as $k\rightarrow\infty$.
Hence, 
\begin{align*}
f=\sum_{k=-\infty}^\infty(g^{k+1}-g^k)
\end{align*}
in $\mathcal D'$ and almost everywhere.
Indeed, since $\mbox{supp}(\sum_i b_i^k)\subset \Omega_k$, then $g^k\rightarrow f$ almost everywhere as 
$k\rightarrow\infty$.
In order to obtain the desired decomposition,
we need the following lemmas from \cite[p. 471]{Tang}.
We denote that a polynomial $P_{ij}^{k+1}$ is an orthogonal
projection of $(f-P_{j})^{k+1}\eta_i^j$ on $\mathcal P_s$,
namely, $P_{ij}^{k+1}$ is the unique polynomial of $\mathcal P_s$ such that, for any $P\in\mathcal P_s$,
\begin{align*}
\int_{\mathbb R^n}\Big(f(x)-P_j^{k+1}(x)\Big)\eta_i^k(x)
P(x)\eta_j^{k+1}(x)dx
=\int_{\mathbb R^n}P_{ij}^{k+1}(x)P(x)\eta_j^{k+1}(x)dx.
\end{align*}

\begin{lem}\label{s4l2}
If $Q_i^{k\ast}\cap Q_i^{(k+1)\ast}\neq  \varnothing,$
then $\ell_j^{k+1}\le 2^4\sqrt{n}\ell_i^k$
and $Q_j^{(k+1)\ast}\subset 2^6nQ_i^{k\ast}
\subset\Omega_k.$ Moreover, there exists
a positive $L$ such that for each $j\in\mathbb N$
the cardinality of 
$$\left\{i\in\mathbb N: Q_i^{k\ast}\cap
Q_j^{(k+1)\ast}\right\}\neq  \varnothing$$ is bounded by $L$.
\end{lem}

\begin{lem}\label{s4l3}
If $0<\ell_j^{k+1}<1$,
$$
\sup_{y\in\mathbb R^n}\left|P_{ij}^{k+1}(y)\eta_j^{k+1}\right|
\le C2^{k+1}.
$$
\end{lem}

\begin{lem}\label{s4l4}
For any $k\in\mathbb Z$,
$$
\sum_{i\in\mathbb N}\left(\sum_{j\in F_2^{k+1}}P_{ij}^{k+1}
\eta_j^{k+1}\right)=0,
$$
where the series converges both in $\mathcal D'$
and pointwisely.
\end{lem}

By Lemma \ref{s4l4} and $\sum_i\eta_i^k=\chi_{\Omega_k}$,
we have
\begin{align*}
g^{k+1}-g^k&=\Big(f-\sum_jb_j^{k+1}\Big)-\Big(f-\sum_ib_i^k\Big)\\
&=\sum_i b_i^k-\sum_j b_j^{k+1}\\
&=\sum_{i\in\mathbb N}\left( b_i^k-\sum_j b_j^{k+1}\eta_i^k+\sum_{j\in F_2^{k+2}}P_{ij}^{k+1}\eta_j^{k+1}\right)\\
&\equiv \sum_{i\in\mathbb N}h_i^k,
\end{align*}
where the series converges both in $\mathcal D'$ and almost everywhere.
For almost everywhere $x\in (\Omega_{k+1})^c$,
\begin{align*}
|f(x)|\le \mathcal G_N(f)(x)\le 2^{k+1}.
\end{align*}
By Lemma \ref{s4l1} and Lemma  
\ref{s4l4}, for all $i\in\mathbb N$
\begin{align*}
\|h_i^k\|_{L^\infty}\le C2^k.
\end{align*}
Next, we consider three cases as follows.\\
{\it Case 1.} When $i\in F_1^k$, we rewrite $h_i^k$ into
\begin{align*}
h_i^k=f\eta_i^k-\sum_{j\in F_1^{k+1}} f\eta_j^{k+1}\eta_i^k
-\sum_{j\in F_2^{k+1}} (f-P_j^{k+1})\eta_j^{k+1}\eta_i^k+\sum_{j\in F_2^{k+2}}P_{ij}^{k+1}\eta_j^{k+1}.
\end{align*}
{\it Case 2.} When $i\in E_1^k\cap F_2^k$, we rewrite $h_i^k$ into
\begin{align*}
h_i^k=(f-P_i^k)\eta_i^k-\sum_{j\in F_1^{k+1}} f\eta_j^{k+1}\eta_i^k
-\sum_{j\in F_2^{k+1}} (f-P_j^{k+1})\eta_j^{k+1}\eta_i^k+\sum_{j\in F_2^{k+2}}P_{ij}^{k+1}\eta_j^{k+1}.
\end{align*}
{\it Case 3.} When $i\in E_2^k$, we rewrite $h_i^k$ into
\begin{align*}
h_i^k=(f-P_i^k)\eta_i^k-\sum_{j\in F_2^{k+1}} (f-P_j^{k+1})\eta_j^{k+1}\eta_i^k+\sum_{j\in F_2^{k+2}}P_{ij}^{k+1}\eta_j^{k+1}.
\end{align*}
In this case, we know that $\sum_{j\in F_1^{k+1}} f\eta_j^{k+1}\eta_i^k=0$. Indeed, when $j\in F_1^{k+1},$
then $\ell_i^k<\frac{1}{2^4n}\ell_j^{k+1}$. Then by Lemma
\ref{s4l1} it follows that $Q_i^{k\ast}\cap Q_j^{(k+1)\ast}=
\varnothing.$ 

Consider Case 1 and Case 2.
$h_i^k$ is supported in $\tilde Q_i^k$ that contains
$Q_i^{k\ast}$ and all the $Q_j^{(k+1)\ast}$ that intersect
$Q_i^{k\ast}$. By Lemma \ref{s4l2}, if $Q_i^{k\ast}\cap Q_j^{(k+1)\ast}\neq \varnothing$, then we have
$\ell_j^{k+1}\le 2^4\sqrt{n}\ell_i^k$
and $Q_j^{(k+1)\ast}\subset 2^6nQ_i^{k\ast}
\subset\Omega_k.$
Fix $\gamma=1+2^{-12-n}$, if $\ell_i^k< 2/(\gamma-1)$,
we choose $\tilde Q_i^k=2^6nQ_i^{k\ast}$.
Otherwise, we can choose  $\tilde Q_i^k=\gamma Q_i^{k\ast}$. Observe that  $\ell_j^{k+1}<1$
for $j\in F_2^{k+1}$, then we get that $Q_j^{(k+1)\ast}\subset
Q(x_i^{k},a(\ell_i^k+2))$ for $j$ fulfilling $Q_i^{k\ast}\cap Q_j^{(k+1)\ast}\neq\varnothing$. 
Hence, $\mbox{supp}(h_i^k) \subset \tilde Q_i^k \subset\Omega_k$.

Consider Case 3. In this case, $i\in E_2^k$ and
$j\in F_2^{k+1}$, then by Lemma \ref{s4l2}, 
 $\mbox{supp}(h_i^k) \subset \tilde Q_i^k\subset \Omega_k$,
 where $\tilde Q_i^k=2^6nQ_i^\ast$.
 Furthermore, we also obtain that $h_i^k$ satisfies the
 moment conditions $\int_{\mathbb R^n}h_i^k(x)q(x)dx=0$
 for any $q\in\mathcal P_s$. Indeed, 
 from the constructions of $P_i^k$ and $P_{ij}^{k+1}$, 
 $(f-P_i^k)\eta_i^k$
 and $\sum_{j\in F_2^{k+1}} (f-P_j^{k+1})\eta_j^{k+1}\eta_i^k
 -\sum_{j\in F_2^{k+2}}P_{ij}^{k+1}\eta_j^{k+1}$ both satisfies
 the moment conditions.

Let $\lambda_{i,k}=C2^k$ and $a_{i,k}=\frac{h_i^k}{\lambda_{i,k}}$.
Then it follows that each $a_{i,k}$ satisfies
$\mbox{supp}\,a_{i,k}\subset \tilde Q_i^k$,
$\|a_{i,k}\|_{L^q}\le |\tilde Q_i^k|^{1/q}$ and 
the desired moment conditions for small cubes with
 \begin{align*}
 f=\sum_{i,k}\lambda_{i,k}a_{i,k}.
 \end{align*}
 
 For convenience, we only need to rearrange $\{a_{i,k}\}$
 and $\{\lambda_{i,k}\}$.
  \begin{align*}
 f=\sum_{i,k}\lambda_{i,k}a_{i,k}\equiv
 \sum_{j=1}^\infty\lambda_{j}a_{j}.
 \end{align*}
 
 To prove the theorem, it remains to show the estimates
 of coefficients, for any $s\in(0,\infty)$
\begin{equation*}
\left\|\left(\sum_{j=1}^\infty(\lambda_j\chi_{Q_j})^{s}\right)^{\frac{1}{s}}\right\|_{L^{p(\cdot)}}
\le C\|f\|_{h^{p(\cdot)}}.
\end{equation*}

To prove it, first observe that
\begin{align*}
&\left\|\left(\sum_{k\in\mathbb Z}\sum_{i\in\mathbb N}(\lambda_i^k\chi_{\tilde Q_i^k})^{s}\right)^{\frac{1}{s}}\right\|_{L^{p(\cdot)}}
\le C\left\|\left\{\sum_{k\in\mathbb Z}\left(
2^k\chi_{\Omega_k}\right)^{s}
\right\}^{\frac{1}{s}}
\right\|_{L^{p(\cdot)}}.
\end{align*}

Moreover, we have the fact that 
$\Omega_{k+1}\subset\Omega_k$ and
$|\bigcap_{k=1}^\infty\Omega_k|=0$,
then for $a.e~x\in\mathbb R^n$. Then it follows that
\begin{align*}
\sum_{k=-\infty}^{\infty}2^{k}\chi_{\Omega_k}(x)
&=\sum_{k=-\infty}^{\infty}2^k\sum_{j=k}^\infty\chi_{\Omega_j\backslash{\Omega_{j+1}}}(x)
=2\sum_{j=-\infty}^{\infty}2^j\chi_{\Omega_j\backslash{\Omega_{j+1}}}(x).
\end{align*}

Therefore, by the definition of $\Omega_j$ we have
\begin{align*}
&\left\|\left\{\sum_{k=-\infty}^\infty\left(
2^k\chi_{ \Omega_j}\right)^{s}
\right\}^{\frac{1}{s}}
\right\|_{L^{p(\cdot)}}\le
C\left\|\left\{\sum_{j=-\infty}^\infty\left(
2^j\chi_{ \Omega_i\setminus\Omega_{j+1}}\right)^{s}
\right\}^{\frac{1}{s}}
\right\|_{L^{p(\cdot)}}\\
&= C\inf\left\{\lambda>0:\int_{\mathbb R^n}\left(
\sum_{j=-\infty}^\infty\frac{2^j\chi_{ \Omega_j\setminus\Omega_{j+1}}}{\lambda}
\right)^{p(x)}dx\le 1\right\}\\
&= C\inf\left\{\lambda>0:\sum_{j=-\infty}^\infty\int_{{ \Omega_j\setminus\Omega_{j+1}}}
\left(\frac{2^j}{\lambda}
\right)^{p(x)}dx\le 1\right\}\\
&\le C\inf\left\{\lambda>0:\int_{\mathbb R^n}
\left(\frac{\mathcal G_{N}f(x)}{\lambda}
\right)^{p(x)}dx\le 1\right\}\leq C\|f\|_{{h}^{p(\cdot)}}.
\end{align*}
Finally, we prove that any $f\in h^{p(\cdot)}$ can be decomposed as in the theorem, since from Remark \ref{s3r2} we learn that
$ h^{p(\cdot)}\cap L^2$ is dense in  $h^{p(\cdot)}$.
Thus, we have completed the proof of Theorem \ref{s4th2}.
$\hfill\Box$

\subsection{Finite atomic decompositions}

The second goal of this section is to discuss that the atomic decomposition norm restricted to finite decompositions is eqivalent to the 
$h^{p(\cdot)}-$norm on some subspace.
For $(p^+\vee 1)<q<\infty$, the function space 
$h_{fin}^{p(\cdot),q}$ is the subspace of $h^{p(\cdot)}$
consisting of all $f$ that have decompositions as finite sums
of the special $(p(\cdot),q)-$atoms. By Corollary \ref{s4c1},
we know that $h_{fin}^{p(\cdot),q}$ is dense in $h^{p(\cdot)}$.

\begin{thm}\label{s4th3}
Let $p(\cdot)\in \mathcal P^0\cap LH$ and $(p^+\vee 1)
<q<\infty$. For $f\in h_{fin}^{p(\cdot),q}$, define
 \begin{align*}
 \|f\|_{h_{fin}^{p(\cdot),q}}\equiv
 \inf\left\{\left\|\sum_{j=1}^M\lambda_j\chi_{Q_j}\right\|_{L^{p(\cdot)}}:\;
f=\sum_{j=1}^M \lambda_ja_j \right\},
\end{align*}
and 
\begin{align*}
 \|f\|_{h_{fin,\ast}^{p(\cdot),q}}\equiv
 \inf\left\{\left\|\left(\sum_{j=1}^M\left(\lambda_j\chi_{Q_j}\right)^{p_-}\right)^{\frac{1}{p_-}}\right\|_{L^{p(\cdot)}}:\;
f=\sum_{j=1}^M \lambda_ja_j \right\},
\end{align*}
where $\{a_j\}_{j=1}^M$ are the special $(p(\cdot),q)-$atoms
and where the infimum is taken over all finite decomposition
of $f$.Then
\begin{align*}
\|f\|_{h^{p(\cdot)}}\sim \|f\|_{h_{fin}^{p(\cdot),q}}\sim \|f\|_{h_{fin,\ast}^{p(\cdot),q}}.
\end{align*}
\end{thm}

\begin{proof}
We only need to prove that 
\begin{align*}
\|f\|_{h^{p(\cdot)}}\sim \|f\|_{h_{fin}^{p(\cdot),q}}.
\end{align*}
First, it is obviously that
\begin{align*}
\|f\|_{h^{p(\cdot)}}\le C\|f\|_{h_{fin}^{p(\cdot),q}}
\end{align*} 
for $f\in h_{fin}^{p(\cdot),q}$. Next we only need to
prove that for any $f\in h_{fin}^{p(\cdot),q}$,
\begin{align*}
\|f\|_{h_{fin}^{p(\cdot),q}}\le C\|f\|_{h^{p(\cdot)}}.
\end{align*} 
By homogeneity we can assume that 
$\|f\|_{h^{p(\cdot)}}=1$. To prove this theorem,
it suffices to show that $\|f\|_{h_{fin}^{p(\cdot),q}}\le C$.
By Theorem \ref{s4th2}, since $f\in h^{p(\cdot)}\cap L^q$,
we can form the following
decomposition of $f$ in terms of  the special local $(p(\cdot),q)-$atoms:
 \begin{align*}
 f=\sum_{i,k}\lambda_{i,k}a_{i,k},
 \end{align*}
where the series converges in $\mathcal D'$ and 
almost everywhere.
If $f\in h_{fin}^{p(\cdot),q}$, then 
$\mbox{supp}\,f\subset Q(x_0, R_0)$ for fixed $x_0\in\mathbb R^n$ and some $R_0\in(1,\infty)$.
We set $\hat Q_0=Q(x_0, \sqrt{n}R_0+2^{3(n+10)+1})$.
Then for any $\psi\in\mathcal D_N$ and $x\in (\hat Q_0)^c$,
for $0<t<1$, we get that
$$
\psi_t\ast f(x)=\int_{Q(x_0,R_0)}\psi_t(x-y)f(y)dy=0.
$$
Then, for any $x\in (\hat Q_0)^c$, it follows that
$x\in (\Omega_k)^c$. Hence, $\Omega_k\subset \hat Q_0$,
and $\mbox{supp} \sum_{i,k}\lambda_{i,k}a_{i,k}\subset \hat Q_0$.
Now we claim that $\sum_{i,k}\lambda_{i,k}a_{i,k}$
 converges to $f$ in $L^q$. In fact,
for any $x\in \mathbb R^n$, we can find a $j\in\mathbb Z$
such that $x\in\Omega_j\setminus\Omega_{j+1}$. 
Since $\mbox{supp}\,a_{i,k}\subset Q_i^{k\ast}\subset \Omega_k\subset\Omega_{j+1}$ for all $k>j$.
Then we have
$$
\left|\sum_{k\in\mathbb Z}\sum_{i\in\mathbb N}
\lambda_{i,k}a_{i,k}\right|\le
\sum_{k\in\mathbb Z}\sum_{i\in\mathbb N}
|\lambda_{i,k}a_{i,k}| \le \sum_{k\le j}2^k\le C2^j
\le C\mathcal G_Nf(x).
$$
By the fact that $\mathcal G_Nf\in L^q$ and the Lebesgue
dominated convergence theorem, we have proved the claim.
For each integer $N>0$, we write $$F_N=\{(i,k):k\in\mathbb Z,i\in\mathbb N,|k|+i\le N\}.$$
Then $f_N\equiv\sum_{(i,k)\in F_N}\lambda_{i,k}a_{i,k}$
is a finite combination of the special local $(p(\cdot),q)-$atoms with
$$\|f_N\|_{h_{fin}^{p(\cdot),q}}=\Bigg\|\sum_{(i,k)\in F_N}\lambda_{i,k}\chi_{Q_{i,k}}\Bigg\|_{L^{p(\cdot)}}
\le C\|f\|_{h^{p(\cdot)}}\le C.$$
Since the series $\sum_{i,k}\lambda_{i,k}a_{i,k}$
converges absolutely in $L^q$, for any given $\epsilon\in(0,\infty)$ there exists $N$ such 
that $\|f-f_N\|_{L^q}<\frac{\epsilon|\hat Q_0|^{1/q}}{\|\chi_{\hat Q_0}\|_{L^{p(\cdot)}}}$. 
Meanwhile, $\mbox{supp}\,f_N\subset \hat Q_0$ together
with the support of $f$ implies that 
$\mbox{supp}\,(f-f_N)\subset \hat Q_0$.
So we can divide $\hat Q_0$ into the union of cubes 
$\{Q_i\}_{i=1}^{N_0}$ with disjoint interior and sidelengths satisfying $\ell_i\in [1,2)$, where $N_0$ depends only on $R_0$ and $n$.
Particularly, we know that
$g_{N,i}\equiv\frac{\|\chi_{ Q_i}\|_{L^{p(\cdot)}}}{\epsilon}(f-f_N)\chi_{Q_i}$ is a special local $(p(\cdot),q)-$atom.
Moreover,
\begin{align*}
\left\|f-f_N\right\|^{p^-}_{h_{fin}^{p(\cdot),q}}&=\left\|\sum_{i=1}^{N_0}\frac{\epsilon}{\|\chi_{Q_i\|_{L^{p(\cdot)}}}}g_{N,i}\right\|^{p^-}_{h_{fin}^{p(\cdot),q}}\\
&=\Bigg\|\sum_{i=1}^{N_0}\frac{\epsilon\chi_{Q_i}}{\|\chi_{Q_i}\|_{L^{p(\cdot)}}}
\Bigg\|^{p^-}_{L^{p(\cdot)}}\\
&\le C\sum_{i=1}^{N_0}\Bigg\|\frac{\epsilon\chi_{Q_i}}{\|\chi_{Q_i}\|_{L^{p(\cdot)}}}
\Bigg\|^{p^-}_{L^{p(\cdot)}}\le C.
\end{align*}
Therefore, we conclude that
$$
f=\sum_{(i,k)\in F_N}\lambda_{i,k}a_{i,k}
+\sum_{i=1}^{N_0}\frac{\epsilon}{\|\chi_{Q_i\|_{L^{p(\cdot)}}}}g_{N,i}
$$
is the desired finite atomic decomposition.
This finished the proof of Theorem \ref{s4th3}.
\end{proof}

\section{Dual spaces of $h^{p(\cdot)}$}
This section is devoted to giving a complete dual theory of $h^{p(\cdot)}$ for $0<p^-\le p^+<\infty$. 
First we consider the case when $0<p^-\le p^+\le 1$
by introducing the local variable Campanato space $bmo^{p(\cdot)}$. Then we establish the local variable Campanato type space $\widetilde{bmo}^{p(\cdot)}$ when we deal with duality of $h^{(\cdot)}$ with $p^+>1$ and $p^-\le1$. 

\subsection{Duality of $h^{p(\cdot)}$ with $p^+\le1$}

In this subsection, we first introduce the local variable Campanato space $bmo^{p(\cdot)}$ and show that the dual space
of $h^{p(\cdot)}$ is $bmo^{p(\cdot)}$. Furthermore,
we also give some equivalent characterizations of the dual of local variable Hardy spaces. We begin with some notions and definitions.

Let $g\in L_{loc}$, and $Q$ be the cube in $\mathbb R^n$.
Let us point out the fact that there exists a unique 
$P\in\mathcal P_d$ with the degree not greater than
$d$, such that 
\begin{align*}
\int_Q(g(x)-P(x))Q(x)dx=0,
\end{align*}
for any $Q\in P_d$. Denote this unique $P$ by
$P_Qg$.

\begin{defn}\label{s5d1}
Suppose that $p(\cdot)\in LH$, $0<p^-\le p^+\leq 1< q< \infty$. 
Let \begin{align*}
&\|f\|_{bmo^{p(\cdot)}}\equiv
\sup_{|Q|<1}
\frac{|Q|}{\|\chi_Q\|_{L^p(\cdot)}}
\left(\frac{1}{|Q|}\int_{Q}|f(x)-P_Qf(x)|^{q'}dx\right)^{1/{q'}}
\\&+
\sup_{|Q|\ge1}
\frac{|Q|}{\|\chi_Q\|_{L^p(\cdot)}}
\left(\frac{1}{|Q|}\int_{Q}|f(x)|^{q'}dx\right)^{1/{q'}},
\end{align*}
where the supreme are taken over all the cubes $Q\subset \mathbb R^n$.
Then the function spaces
$$
bmo^{p(\cdot)}(\mathbb R^n)=\{f\in L_{loc}:
\|f\|_{bmo^{p(\cdot)}}<\infty\}
$$
are called the local variable Campanato spaces.
\end{defn}

We also introduce the local variable Lipschitz spaces 
as follows.
\begin{defn}\label{s5d2}
Suppose that $p(\cdot)\in LH$, $0<p^-\le p^+\leq 1< q< \infty$. 
Let \begin{align*}
\|f\|_{lip_{p(\cdot)}}\equiv
\sup_{|Q|<1}\sup_{x,y\in Q}
\frac{|Q||f(x)-f(y)|}{\|\chi_Q\|_{L^p(\cdot)}}+
\sup_{|Q|\ge1}
\frac{|Q|\|f\|_{L^\infty(Q)}}{\|\chi_Q\|_{L^p(\cdot)}},
\end{align*}
where the first supreme are taken over all cubes 
$Q\subset \mathbb R^n$ with $|Q|<1$ and
the second supreme are taken over all cubes 
$Q\subset \mathbb R^n$ with $|Q|\ge1$.
Then the local variable Lipschitz spaces is defined by setting
$$
lip_{p(\cdot)}(\mathbb R^n)=\{f\in L_{loc}:
\|f\|_{bmo^{p(\cdot)}}<\infty\}.
$$
\end{defn}

\begin{thm}\label{s5th1}
Suppose that $p(\cdot)\in LH$, $0<p^-\le p^+\le1<
q<\infty$. 
The dual space of $h^{p(\cdot)}$
is $bmo^{p(\cdot)}$ in the following sense:\\
\noindent (1) For any $g\in bmo^{p(\cdot)}$, the linear functional $l_g$,
defined initially on $h_{fin}^{p(\cdot),q}$ has a unique extension to $h^{p(\cdot)}$ with $\|l_g\|\le C\|g\|_{bmo^{p(\cdot)}}$.\\
\noindent (2) Conversely, for any $l\in(h^{p(\cdot)})'$, 
there exists a unique function $g\in bmo^{p(\cdot)}$ 
such that 
$l_g(f)=\big<g,f\big>$ holds ture for all $f\in h_{fin}^{p(\cdot),q}$ with
$\|g\|_{bmo^{p(\cdot)}}\le C\|l\|$.
\end{thm}

\begin{proof}
Let $g\in bmo^{p(\cdot)}$ and $f\in h_{fin}^{p(\cdot),q}$.
Then for some numbers $\{\lambda_j\}_{j=1}^M$
and the special local $(p(\cdot),q)-$atoms, we have
\begin{align*}
f=\sum_{j=1}^M \lambda_ja_j
\end{align*}
in $\mathcal D'$.
When $|Q_j|<1$, by the cancellation of $a_j$
and H\"older's inequality, we get that
\begin{align*}
\left|\int_{\mathbb R^n}f(x)g(x)dx\right|
&=\left|\sum_{j=1}^M\lambda_j\int_{\mathbb R^n}a_j(x)g(x)\right|\\
&\le\sum_{j=1}^M|\lambda_j|\left|\int_{\mathbb R^n}a_j(x)
g(x)\chi_{Q_j}dx\right|\\
&\le\sum_{j=1}^M|\lambda_j|\left|\int_{\mathbb R^n}a_j(x)
(g(x)-P_{Q_j}g(x))\chi_{Q_j}dx\right|\\
&\le\sum_{j=1}^M|\lambda_j|\|a_j\|_{L^q}\left(\int_{\mathbb R^n}|g(x)-P_{Q_j}g(x)|^{q'}\chi_{Q_j}dx\right)^{q'}\\
&=\sum_{j=1}^M|\lambda_j|\|\chi_{Q_j}\|_{L^{p(\cdot)}}\frac{|Q_j|}{\|\chi_{Q_j}\|_{L^{p(\cdot)}}}
\left(\frac{1}{|Q_j|}\int_{Q_j}|g(x)-P_{Q_j}g(x)|^{q'}dx\right)^{q'}\\
&\le \left(\sum_{j=1}^M|\lambda_j|\|\chi_{Q_j}\|_{L^{p(\cdot)}}\right)\|g\|_{bmo^{p(\cdot)}}.
\end{align*}
We claim that 
\begin{align*}
\sum_{j=1}^M|\lambda_j|\|\chi_{Q_j}\|_{L^{p(\cdot)}}\le
C\|f\|_{h_{fin,\ast}^{p(\cdot),q}}.
\end{align*}
Write $\kappa=\sum_{j}^M|\lambda_j|\|\chi_{Q_j}\|_{L^{p(\cdot)}}$. Indeed, since $p^+\le 1$, then we obtain that
\begin{align*}
&\int_{\mathbb R^n}\left(\sum_{i=1}^M\left(\frac{\lambda_i\chi_{Q_i}(x)}{\kappa}\right)^{p_-}\right)^{\frac{p(x)}{p_-}}dx\\
&\ge \int_{\mathbb R^n}\sum_{i=1}^M\left(\frac{\lambda_i\chi_{Q_i}(x)}{\kappa}\right)^{p(x)}dx\\
&=\sum_{i=1}^M\int_{\mathbb R^n}\left(\frac{\lambda_i\chi_{Q_i}(x)}{\kappa}\right)^{p(x)}dx\\
&\ge\sum_{i=1}^M\frac{\lambda_i\|\chi_{Q_i}\|_{L^{p(\cdot)}}}{\kappa}\int_{\mathbb R^n}\left(\frac{\chi_{Q_i}(x)}{\|\chi_{Q_i}\|_{L^{p(\cdot)}}}\right)^{p(x)}dx\\
&=\sum_{i=1}^M\frac{\lambda_i\|\chi_{Q_i}\|_{L^{p(\cdot)}}}{\kappa}=1.
\end{align*}

Thus, if $|Q_j|<1$, then we get that
\begin{align*}
\left|\int_{\mathbb R^n}f(x)g(x)dx\right|
&\le\|f\|_{h_{fin,\ast}^{p(\cdot),q}}\|g\|_{bmo^{p(\cdot)}}.
\end{align*}

Similarly, if $|Q_j|\ge 1$, then we obtain that
\begin{align*}
\left|\int_{\mathbb R^n}f(x)g(x)dx\right|
&=\left|\sum_{j=1}^M\lambda_j\int_{\mathbb R^n}a_j(x)g(x)\right|\\
&\le\sum_{j=1}^M|\lambda_j|\left|\int_{\mathbb R^n}a_j(x)
g(x)\chi_{Q_j}dx\right|\\
&\le\sum_{j=1}^M|\lambda_j|\|a_j\|_{L^q}\left(\int_{{Q_j}}|g(x)|^{q'}dx\right)^{q'}\\
&=\sum_{j=1}^M|\lambda_j|\|\chi_{Q_j}\|_{L^{p(\cdot)}}
\frac{|Q_j|}{\|\chi_{Q_j}\|_{L^{p(\cdot)}}}
\left(\frac{1}{|Q_j|}\int_{Q_j}|g(x)|^{q'}dx\right)^{q'}\\
&\le \|f\|_{h_{fin,\ast}^{p(\cdot),q}}\|g\|_{bmo^{p(\cdot)}}.
\end{align*}
By Theorem \ref{s4th3}, we conclude that
\begin{align*}
l_g(f):=\big<g,f\big>
\le \|f\|_{h^{p(\cdot)}}\|g\|_{bmo^{p(\cdot)}}.
\end{align*}
This show that $l_g$ can be extended uniquely to 
a bounded linear functional $h^{p(\cdot)}$ with
$$\|l_g\|\le C\|g\|_{bmo^{p(\cdot)}}.$$
Thus, the former conclusion of the theorem holds.

Now we prove $(2)$.
To prove it, we first claim that $\|a\|_{h^{p(\cdot)}}\le C\|\chi_Q\|_{L^{p(\cdot)}}$ for every special local $(p(\cdot),q)-$atom supported on a cube $Q$.
Indeed, let $
\varphi\in\mathcal D$ 
be a nonnegative and radial function supported on
$Q(0,1/2)$ with $\int \varphi(x)dx\neq 0$.
Applying Lemma \ref{s2l7}, the boundedness of the maximal operator
$M$ and the fact that $\mathcal M_\varphi a(x)\le 
CMa(x)$ yield that
\begin{align*}
\|(\mathcal M_\varphi a)\chi_{\tilde Q}\|_{L^{p(\cdot)}}&\le \|(Ma)\chi_{\tilde Q}\|_{L^{p(\cdot)}}\\
&\le\|Ma\|_{L^q}\|\chi_{\tilde Q}\|_{L^{\tilde q(\cdot)}}\\
&\le\|a\|_{L^q}\|\chi_{\tilde Q}\|_{L^{\tilde q(\cdot)}}\\
&\le|Q|^{\frac{1}{q}}\|\chi_{\tilde Q}\|_{L^{\tilde q(\cdot)}}\le \|\chi_Q\|_{L^{p(\cdot)}},
\end{align*}
where  $\tilde q(\cdot)$ is defined by
$\frac{1}{p(x)}=\frac{1}{q}+\frac{1}{\tilde q(x)}.$
Next we need to show that 
$\|(\mathcal M_\varphi a)\chi_{(\tilde Q)^c}\|_{L^{p(\cdot)}}\le \|(Ma)\chi_{\tilde Q}\|_{L^{p(\cdot)}}\le \|\chi_Q\|_{L^{p(\cdot)}}$.
When $|Q|\le1$, by following the standard argument
in \cite[p. 3682]{NS}, we obtain that
\begin{align*}
  |(a\ast \varphi_j)(x)| \le
C\frac{\ell(Q)^{n+d+1}}{|x-c_Q|^{n+d+1}},
\end{align*}
where $x\in (\tilde Q)^c$ and $c_Q$ is the center
of the cube $Q$.
Hence, we get that 
\begin{align*}
\|(\mathcal M_\varphi a)\chi_{(\tilde Q)^c}\|_{L^{p(\cdot)}}
&\le C\left\|
\frac{\ell(Q)^{n+d+1}}{|\cdot-c_Q|^{n+d+1}}
\right\|_{L^{p(\cdot)}}\\
&\le C\left\|
(M\chi_{\tilde Q})^{n+d+1/n}
\right\|_{L^{p(\cdot)}}\le C\|\chi_Q\|_{L^{p(\cdot)}}.
\end{align*}
When $|Q|>1$, observe that $j\ge0$,
for any special local $(p(\cdot),q)-$atom 
$a$ with $\mbox{supp}\;a\subset Q$ and $x\in (\tilde Q)^c$, we have
\begin{align*}
  |(a\ast \varphi_j)(x)| &\le
  \int_{Q}|a(y)\varphi_j(x-y)dy \\
  &\le \sup_{y\in Q}|\varphi_j(x-y)\int_{Q}|a(y)|dy\\
  &\le C\frac{2^{jn}}{(1+2^j|x-c_Q|)^M}\|a\|_{L^q}|Q|^{\frac{1}{q'}}\\
  &\le C\frac{2^{jn}}{(1+2^j|x-c_Q|)^M}\\
  &\le C\frac{2^{j(n-M)}(l(Q))^M}{|x-c_Q|^M}
  \le C\frac{(l(Q))^M}{|x-c_Q|^M},
\end{align*}
for any sufficient large $M>n>0$.
We choose $M$ such that $\frac{Mp^-}{n}>1$.
Similarly, we can obtain that 
\begin{align*}
\|(\mathcal M_\varphi a)\chi_{(\tilde Q)^c}\|_{L^{p(\cdot)}}\le C\left\|
(M\chi_{\tilde Q})^{M/n}
\right\|_{L^{p(\cdot)}}\le C\|\chi_Q\|_{L^{p(\cdot)}}.
\end{align*}
Therefore, we have proved the claim.
Fix a cube $Q$ with $\ell(Q)\ge 1$.
For any given $f\in L^q(Q)$ with $\|f\|_{L^q(Q)}>0$,
set
\begin{align*}
a(x)\equiv \frac{f(x)\chi_{Q}|Q|^{1/q}}{\|f\|_{L^q(Q)}}.
\end{align*}
Then $a$ is obviously a special local $(p(\cdot),q)-$atom. Thus, $|l(a)|\le \|l\|\|a\|_{h^{p(\cdot)}}
\le \|l\|\|\chi_Q\|_{L^{p(\cdot)}}.$
It follows that for any $l\in(h^{p(\cdot)})'$,
\begin{align*}
|l(f)|\le
\|l\|\|f\|_{h^{p(\cdot)}}\le
\|l\|\|f\|_{L^q(Q)}\|\chi_Q\|_{L^{p(\cdot)}}|Q|^{-1/q}.
\end{align*}
Hence, $l\in(L^q(Q))'$ and $(h^{p(\cdot)})'\subset
(L^q(Q))'$.
Since $1<q<\infty$, using the duality $L^q(Q)-L^{q'}(Q)$, we find that there exists a  $g^Q\in L^{q'}(Q)$
such that for all $f\in L^q(Q)$,
\begin{align*}
lf=\int_Qf(x)g^Q(x)dx,
\end{align*}
and $\|g^Q\|_{L^{q'}(Q)}\le \|l\|\|\chi_Q\|_{L^{p(\cdot)}}|Q|^{-1/q}$.
Take a sequence $\{Q_j\}_{j\in\mathbb N}$ of cubes
such that $Q_j\subset Q_{j+1}$, $\cup_{j\in\mathbb N}=\mathbb R^n$ and $\ell(Q_1)\ge 1$.
Similarly, we know that there exists a  $g^{Q_j}\in L^{q'}(Q_j)$ such that
for each $Q_j$,
\begin{align*}
lf=\int_{Q_j}f(x)g^{Q_j}(x)dx,
\end{align*}
and $\|g^{Q_j}\|_{L^{q'}(Q-j)}\le \|l\|\|\chi_{Q_j}\|_{L^{p(\cdot)}}|Q_j|^{-1/q}$.
Then we can construct a function $g$ such that
for all $f\in L^q(Q_j)$
\begin{align*}
lf=\int_{Q_j}f(x)g(x)dx.
\end{align*}
Assume that $f\in L^q(Q_1)$. We have that 
there exists a  $g^{Q_1}\in L^{q'}(Q_1)$ such that,
\begin{align*}
lf=\int_{Q_1}f(x)g^{Q_1}(x)dx.
\end{align*}
Observe that $f\in L^q(Q_1)\subset L^q(Q_2)$
and it follows that there exists a  $g^{Q_2}\in L^{q'}(Q_2)$ such that
\begin{align*}
lf=\int_{Q_1}f(x)g^{Q_1}(x)dx=\int_{Q_2}f(x)g^{Q_2}(x)dx.
\end{align*}
Therefore, for all $f\in L^q(Q_1)$,
\begin{align*}
\int_{Q_1}f(x)(g^{Q_1}(x)-g^{Q_2}(x))dx=0,
\end{align*}
which implies that $g^{Q_1}(x)=g^{Q_2}(x)$, 
where $x\in Q_1$. 
Setting $g(x)=g^{Q_1}(x)$ when $x\in Q_1$
and $g(x)=g^{Q_2}(x)$ when $x\in Q_2\setminus Q_1$. Then for all $f\in L^q(Q_j)$ with $j=1,2$, we have 
\begin{align*}
lf=\int_{Q_j}f(x)g(x)dx.
\end{align*}
Repeating the similar argument, we can obtain a $g(x)$ such that the above equality holds for all $j\in\mathbb N$, which implies that
\begin{align*}
lf=\big<g,f\big>=\int_{\mathbb R^n} f(x)g(x)dx
\end{align*}
also holds for all $f\in h_{fin}^{p(\cdot),q}$.
Choose any $Q\subset \mathbb R^n$ with $\ell(Q)\ge1$, $\|f\|_{L^q}\le1$, and $\mbox{supp} f\subset Q$. We write $a\equiv C|Q|^{1/q}f(x)\chi_Q(x)$
with a suitable constant $C$. Then $a$ is a special
local $(p(\cdot),q)-$ atom. Then by the equality $l_a=\int_Qa(x)g(x)dx$ and $l\in (h^{p(\cdot)})'$,
we obtain that
\begin{align*}
|la|=\left|\int_Qa(x)g(x)dx\right|\le\|l\|\|\chi_Q\|_{L^p(\cdot)}.
\end{align*}
The above inequality implies that
\begin{align*}
\|\chi_Q\|^{-1}_{L^p(\cdot)}|Q|^{-1/q}\left|\int_Qf(x)g(x)dx\right|\le\|l\|.
\end{align*}
It follows that
\begin{align*}
\frac{|Q|}{\|\chi_Q\|_{L^p(\cdot)}}
\left(\frac{1}{|Q|}\int_{Q}|g(x)|^{q'}dx\right)^{1/{q'}}
\le \|l\|.
\end{align*}
When $|Q|\le 1$, by using the similar argument
in \cite[p. 3724-p. 3725]{NS}, we can deduce that
\begin{align*}
\frac{|Q|}{\|\chi_Q\|_{L^p(\cdot)}}
\left(\frac{1}{|Q|}\int_{Q}|g(x)-P_Qg(x)|^{q'}dx\right)^{1/{q'}}
\le \|l\|.
\end{align*}
Therefore, we have shown that $g(x)$ belongs to $bmo^{p(\cdot)}$.  This finishes the proof of 
Theorem \ref{s5th1}.
\end{proof}

Repeating the almost same argument in the proof of
Theorem \ref{s5th1}, we immediately deduce the following theorem and we omit the details.  

\begin{thm}\label{s5th2}
Suppose that $p(\cdot)\in LH$, $0<p^-\le p^+\le1<
q<\infty$. 
The dual space of $h^{p(\cdot)}$
is $lip_{p(\cdot)}$ in the following sense:\\
\noindent (1) For any $g\in lip_{p(\cdot)}$, the linear functional $l_g$,
defined initially on $h_{fin}^{p(\cdot),q}$ has a unique extension to $h^{p(\cdot)}$ with $\|l_g\|\le C\|g\|_{lip_{p(\cdot)}}$.\\
\noindent (2) Conversely, for any $l\in(h^{p(\cdot)})'$, 
there exists a unique function $g\in lip_{p(\cdot)}$ 
such that 
$l_g(f)=\big<g,f\big>$ holds ture for all $f\in h_{fin}^{p(\cdot),q}$ with
$\|g\|_{lip_{p(\cdot)}}\le C\|l\|$.
\end{thm}

Next, we also state the Carleson measure characterization for the dual of $h^{p(\cdot)}$.
We see that the local variable
Carleson measure space $cmo^{p(\cdot)}$
is the dual space of the local variable Hardy space $h^{p(\cdot)}$ in \cite{TWL}. We need some notations.
Denote by $\ell(Q)=2^{-j}$ the side length of $Q=Q_{j{\bf k}}$, $\bf k\in\mathbb Z^n$.
Denote by $z_Q=2^{-j}{\bf k}$
the left lower corner of $Q$
and by $x_Q$ is any point in $Q$ when $Q=Q_{j{\bf k}}$.
Denote $\Pi_{j}=\{Q:
Q=Q_{j{\bf k}}\}$ and $\Pi=
\cup_{j\in\mathbb{N}}\Pi_{j}$.
For any function
$\psi$ defined on $\mathbb R^n,$ $j\in\mathbb Z$, and $Q=Q_{j{\bf k}}$, set
\begin{align*}
\psi_j(x)=2^{jn}\psi(2^{j}x),\quad \psi_Q(x)=|Q|^{1/2}\psi_j({x-z_Q}).
\end{align*}
For more detail on Carleson measure spaces, please see \cite{HHL1,HHLT,HLW,LL,T19}.
\begin{defn}\label{cp}
The local variable Carleson measure space
$cmo^{p(\cdot)}(\mathbb R^n)$ is the collection of all $f\in \mathcal D'$ fulfilling
$$
\|f\|_{cmo^{p(\cdot)}}\equiv\sup_{P\in\Pi}\left\{\frac{|P|}{\|\chi_P\|^2_{p(\cdot)}}
\int_{P}\sum_{j\in\mathbb N}\sum_{Q\in \Pi_j,
\;Q\subset
P}|Q|^{-1}|\left<f, \psi_Q\right>|^2\chi_{Q}(x)dx\right\}^{1/2}<\infty.
$$
\end{defn}

\begin{thm}\cite{TWL}\label{s5th3}
Suppose that $p(\cdot)\in LH$, $0<p^-\le p^+\leq1$. The dual space of $h^{p(\cdot)}$
is $cmo^{p(\cdot)}$ in the following sense.\\
\noindent (1) For $g\in cmo^{p(\cdot)}$, the linear functional $l_g$,
defined initially on $\mathcal D$, extends to a continuous linear functional
on $h^{p(\cdot)}$ with $\|l_g\|\le C\|g\|_{cmo^{p(\cdot)}}$.

\noindent (2) Conversely, every continuous linear functional $l$ on $h^{p(\cdot)}$
satisfies $l=l_g$ for some $g\in cmo^{p(\cdot)}$ with
$\|g\|_{cmo^{p(\cdot)}}\le C\|l\|$.
\end{thm}

From Theorems \ref{s5th1}, \ref{s5th2} and \ref{s5th3}, we immediately deduce the following equivalent definitions of the dual of local variable Hardy spaces.
\begin{cor}\label{s5c1}
Let $p(\cdot)\in LH$ and $0<p^-\le p^+\leq1<q<\infty$.
Then local variable Campanato space $bmo^{p(\cdot)}$, local variable Lipschitz spaces $lip_{p(\cdot)}$ and local variable Carleson measure spaces $cmo^{p(\cdot)}$
coincide as sets and
$$\|f\|_{bmo^{p(\cdot)}}\sim\|f\|_{lip_{p(\cdot)}}
\sim\|f\|_{cmo^{p(\cdot)}}.$$
\end{cor}

\subsection{Duality of $h^{p(\cdot)}$ with $0<p^-\le p^+<\infty$}

In this subsection, we introduce a kind of local variable Campanato type space $\widetilde{bmo}^{p(\cdot)}$. Inspired by \cite{HuYY}, we show that the dual space of $h^{p(\cdot)}$ is $\widetilde{bmo}^{p(\cdot)}$, for all $p(\cdot)\in LH$ fulfilling $0<p^-\le p^+<\infty$.  To prove it, we also need to prove that
a kind of variable Campanato type space $\widetilde{BMO}^{p(\cdot)}$ is the dual space of the global variable Hardy space $H^{p(\cdot)}$, which gives a complete answer to the open question proposed by Izuki et al. in \cite{INS}.
 We begin with some definitions.
The variable Campanato type space $\widetilde{BMO}^{p(\cdot)}$ is defined as follows.
\begin{defn}\label{s5d4}
Suppose that $p(\cdot)\in LH$, $0<p^-\le p^+<\infty$ and $\leq 1< q< \infty$. 
Let 
\begin{align*}
\|f\|_{\widetilde{BMO}^{p(\cdot)}}\equiv
\sup\left\|\sum_{i=1}^M\lambda_i\chi_{Q_i}\right\|^{-1}_{L^{p(\cdot)}}
\sum_{j=1}^M\left\{
{\lambda_j|Q_j|}\left(\frac{1}{|Q_j|}\int_{Q_j}|f(x)-P_{Q_j}f(x)|^{q'}dx\right)^{1/{q'}}\right\},
\end{align*}
where the supreme are taken over all 
$M\in\mathbb N$, the cubes $Q_j\subset \mathbb R^n$, and non-negative numbers $\{\lambda_j\}_{j=1}^M$
satisfying $\sum_{j=1}^M\lambda_j
\|\chi_{Q_j}\|_{L^{p(\cdot)}}\neq 0$.
Then the function spaces
$$
\widetilde{BMO}^{p(\cdot)}(\mathbb R^n)=\{f\in L_{loc}:
\|f\|_{\widetilde{BMO}^{p(\cdot)}}<\infty\}
$$
are called the variable Campanato type spaces.
\end{defn}

%\begin{rem}  Recall that the Campanato spaces with variable growth conditions $BMO^{p(\cdot)}$ is introduced in \cite{NS} with the quasi-norm \begin{align*} \|f\|_{BMO^{p(\cdot)}}\equiv\sup_{Q\subset\mathbb R^n}\frac{|Q|}{\|\chi_Q\|_{L^p(\cdot)}}\left(\frac{1}{|Q|}\int_{Q}|f(x)-P_Qf(x)|^{q'}dx\right)^{1/{q'}}.\end{align*}\end{rem}

\begin{defn}\label{s5d5}
Suppose that $p(\cdot)\in LH$, $0<p^-\le p^+<\infty$ and $\leq 1< q< \infty$. 
Let \begin{align*}
&\|f\|_{\widetilde{bmo}^{p(\cdot)}}\equiv
\sup_{|Q|<1}\left\|\sum_{i=1}^M\lambda_i\chi_{Q_i}\right\|^{-1}_{L^{p(\cdot)}}
\sum_{j=1}^M\left\{
{\lambda_j|Q_j|}\left(\frac{1}{|Q_j|}\int_{Q_j}|f(x)-P_{Q_j}f(x)|^{q'}dx\right)^{1/{q'}}\right\}
\\&+
\sup_{|Q|\ge1}\left\|\sum_{i=1}^M\lambda_i\chi_{Q_i}\right\|^{-1}_{L^{p(\cdot)}}
\sum_{j=1}^M\left\{
{\lambda_j|Q_j|}
\left(\frac{1}{|Q|}\int_{Q}|f(x)|^{q'}dx\right)^{1/{q'}}\right\},
\end{align*}
where the first supreme are taken over all 
$M\in\mathbb N$, the cubes $Q_j\subset \mathbb R^n$ with $|Q|<1$, and non-negative numbers $\{\lambda_j\}_{j=1}^M$
satisfying $\sum_{j=1}^M\lambda_j
\|\chi_{Q_j}\|_{L^{p(\cdot)}}\neq 0$
and the second supreme are taken over all 
$M\in\mathbb N$, the cubes $Q_j\subset \mathbb R^n$ with $|Q|\ge1$, and non-negative numbers $\{\lambda_j\}_{j=1}^M$
satisfying $\sum_{j=1}^M\lambda_j
\|\chi_{Q_j}\|_{L^{p(\cdot)}}\neq 0$.
Then the function spaces
$$
\widetilde{bmo}^{p(\cdot)}(\mathbb R^n)
=\{f\in L_{loc}:
\|f\|_{\widetilde{bmo}^{p(\cdot)}}<\infty\}
$$
are called the local variable Campanato type spaces.
\end{defn}

Next, we obtain the duality between $H^{p(\cdot)}(\mathbb R^n)$ and $\widetilde{BMO}^{p(\cdot)}(\mathbb R^n)$ with $0<p^-\le p^+<\infty$.
\begin{thm}\label{s5th4}
Suppose that $p(\cdot)\in LH\cap \mathcal P^0$, $(p^+\vee 1)<
q<\infty$. 
The dual space of $H^{p(\cdot)}$
is $\widetilde{BMO}^{p(\cdot)}$ in the following sense:\\
\noindent (1) Let $g\in \widetilde{BMO}^{p(\cdot)}$.
Then the linear functional $L_g: f\rightarrow L_g(f)=
\big<g,f\big>$,
defined initially on $H_{fin}^{p(\cdot),q}$ has a bounded extension to $H^{p(\cdot)}$ with $\|L_g\|\le C\|g\|_{\widetilde{BMO}^{p(\cdot)}}$.\\
\noindent (2) Conversely, for any $L\in(H^{p(\cdot)})'$, 
there exists a unique function $g\in\widetilde{BMO}^{p(\cdot)}$ such that 
$L_g(f)=\big<g,f\big>$ holds ture for all $f\in H_{fin}^{p(\cdot),q}$ with
$\|g\|_{\widetilde{BMO}^{p(\cdot)}}\le C\|L\|$.
\end{thm}

\begin{proof}
To prove Theorem \ref{s5th4}, we first recall some known results on atom decomposition
characterizations of $H^{p(\cdot)}$, which can be found in \cite[Proposition 2.4]{CMN} (also see \cite[Theorem 7.8]{CW}).
Fix $(p^+\vee 1)<q<\infty$.
Given countable collections of cubes $\{Q_j\}_{j=1}^\infty$,
of non-negative coefficients  $\{\lambda_j\}_{j=1}^\infty$ and of the special $(p(\cdot),q)$-atoms $\{a_j\}_{j=1}^\infty$.
If 
\begin{equation*}
\left\|\sum_{j=1}^\infty\lambda_j\chi_{Q_j}\right\|_{L^{p(\cdot)}}<\infty.
\end{equation*}
Then the series $\sum_{j=1}^\infty \lambda_ja_j$
converges in $H^{p(\cdot)}$ and satisfies
\begin{equation*}
\left\|\sum_{j=1}^\infty \lambda_ja_j\right\|_{H^{p(\cdot)}}\le C\left\|\sum_{j=1}^\infty\lambda_j\chi_{Q_j}\right\|_{L^{p(\cdot)}}
\end{equation*}
Furthermore, 
the finite atomic variable Hardy spaces
$H_{fin}^{p(\cdot),q}$ is defined to be be the
set of all $f$ fulfilling that there exists an $M\in\mathbb N$, such that 
$f=\sum_{j=1}^M \lambda_ja_j$ with
the quasi-norm
 \begin{align*}
 \|f\|_{H_{fin}^{p(\cdot),q}}\equiv
 \inf\left\|\sum_{j=1}^M\lambda_j\chi_{Q_j}\right\|_{L^{p(\cdot)}}<\infty,
\end{align*}
where the infimum is taken over all decompositions of 
$f$ as above. Then $\|f\|_{H_{fin}^{p(\cdot),q}}$ and $\|f\|_{H^{p(\cdot)}}$ are equivalent
quasi-norms on $H_{fin}^{p(\cdot),q}$.
Now we prove $(1)$. Let $g\in \widetilde{BMO}^{p(\cdot)}$ and $f\in H_{fin}^{p(\cdot),q}$.
Then for some numbers $\{\lambda_j\}_{j=1}^M$
and the special $(p(\cdot),q)-$atoms, we have
$
f=\sum_{j=1}^M \lambda_ja_j
$.
By the atomic decomposition results of $H^{p(\cdot)}$, the cancellation of $a_j$, H\"older's inequality
and the size condition of $a_j$, we get that
\begin{align*}
|L_g(f)|&=\left|\int_{\mathbb R^n}f(x)g(x)dx\right|\\
&\le\sum_{j=1}^M\lambda_j\left|\int_{\mathbb R^n}a_j(x)\left[g(x)-P_{Q_j}g(x)\right]\right|\\
&\le\sum_{j=1}^M\lambda_j\|a_j\|_{L^q}\left(\int_{\mathbb R^n}|g(x)-P_{Q-j}g(x)|^{q'}\chi_{Q_j}dx\right)^{q'}\\
&=\sum_{j=1}^M|\lambda_j|{|Q_j|}
\left(\frac{1}{|Q_j|}\int_{Q_j}|g(x)-P_Qg(x)|^{q'}dx\right)^{q'}\\
&\le \left\|\sum_{i=1}^M\lambda_i\chi_{Q_i}\right\|_{L^{p(\cdot)}}\|g\|_{\widetilde{BMO}^{p(\cdot)}}\\
&\le \|f\|_{H_{fin}^{p(\cdot),q}}\|g\|_{\widetilde{BMO}^{p(\cdot)}}\sim  \|f\|_{H^{p(\cdot)}}\|g\|_{\widetilde{BMO}^{p(\cdot)}},
\end{align*}
which implies that $(1)$ holds true.
It remains to be proved $(2)$.
For any $M\in\mathbb N$, any cubes $Q_j\subset \mathbb R^n$, and non-negative numbers $\{\lambda_j\}_{j=1}^M$
with $\sum_{j=1}^M\lambda_j
\|\chi_{Q_j}\|_{L^{p(\cdot)}}\neq 0$,
let $f_j\in L^q(Q_j)$ with $\|f_j\|_{L^q(Q_j)}=1$
satisfying
$$
\left[\int_{Q_j}\left|g(x)-P_{Q_j}g(x)\right|^{q'}dx\right]^{1/{q'}}=
\int_{Q_j}\left[g(x)-P_{Q_j}g(x)\right]f_j(x)dx
$$
and, for any $x\in\mathbb R^n$, define
\begin{align*}
a_j(x)\equiv \frac{|Q_j|^{1/q}(f_j(x)-P_{Q_j}f_j(x))\chi_{Q_j}}{\|f_j-P_{Q_j}f_j(x)\|_{L^q(Q_j)}}.
\end{align*}
Then from the definition of the atom, it follows that $a_j$ is a special $(p(\cdot),q)-$atom
and $\sum_{j=1}^M\lambda_ja_j\in H^{p(\cdot)}$.
For any $L\in (H^{p(\cdot)})'$, repeating the similar argument to that used in the proof of Theorem
\ref{s5th1}, we find that there exists a unique 
$g\in\widetilde{BMO}^{p(\cdot)}$ such that 
$$L_g(f)=\int_{\mathbb R^n}f(x)g(x)dx.$$
In fact,
if $L\in (H^{p(\cdot)})'$, then we know that
$$L\left(\sum_{j=1}^M\lambda_ja_j\right)\le \|L\|
\left\|\sum_{j=1}^M\lambda_ja_j\right\|_{H^{p(\cdot)}}
\le \left\|L\right\|
\left\|\sum_{j=1}^M\lambda_j\chi_{Q_j}\right\|_{L^{p(\cdot)}}.$$
Hence, from this and the fact that $\|P_{Q_j}f_j(x)\|_{L^q}\le C\|f_j\|_{L^q},$ we conclude that
\begin{align*}
&\sum_{j=1}^M
{\lambda_j|Q_j|}\left(\frac{1}{|Q_j|}\int_{Q_j}|g(x)-P_{Q_j}g(x)|^{q'}dx\right)^{1/{q'}}\\
&= \sum_{j=1}^M
{\lambda_j|Q_j|^{1/q}}\int_{Q_j}\left[g(x)-P_{Q_j}g(x)\right]f_j(x)dx\\
&= \sum_{j=1}^M
{\lambda_j|Q_j|^{1/q}}\int_{Q_j}\left[f_j(x)-P_{Q_j}f_j(x)\right]g(x)\chi_{Q_j}dx\\
&\le C\sum_{j=1}^M
\lambda_j\int_{Q_j}a_j(x)g(x)dx\\
&\sim \sum_{j=1}^M\lambda_jL\left(a_j\right)
\sim L\left(\sum_{j=1}^M\lambda_ja_j\right)\\
&\le C\|L\|
\left\|\sum_{j=1}^M\lambda_j\chi_{Q_j}\right\|_{L^{p(\cdot)}},
\end{align*}
which implies that
$g\in\widetilde{BMO}^{p(\cdot)}$ with
$\|g\|_{\widetilde{BMO}^{p(\cdot)}}\le C\|L\|$.
Therefore, we have completed the proof of Theorem
\ref{s5th4}.
\end{proof}

From Theorem \ref{s5th1} and
Theorem \ref{s5th4}, by applying nearly identical method to the above proofs, we can deduce the duality
between $h^{p(\cdot)}(\mathbb R^n)$ and $\widetilde{bmo}^{p(\cdot)}(\mathbb R^n)$ with $0<p^-\le p^+<\infty$.
\begin{thm}\label{s5th4}
Suppose that $p(\cdot)\in LH\cap \mathcal P^0$, $(p^+\vee 1)<
q<\infty$. 
The dual space of $h^{p(\cdot)}$
is $\widetilde{bmo}^{p(\cdot)}$ in the following sense:\\
\noindent (1) Let $g\in \widetilde{bmo}^{p(\cdot)}$.
Then the linear functional $L_g: f\rightarrow l_g(f)=
\big<g,f\big>$,
defined initially on $h_{fin}^{p(\cdot),q}$ has a bounded extension to $h^{p(\cdot)}$ with $\|l_g\|\le C\|g\|_{\widetilde{bmo}^{p(\cdot)}}$.\\
\noindent (2) Conversely, for any $L\in(h^{p(\cdot)})'$, 
there exists a unique function $g\in\widetilde{bmo}^{p(\cdot)}$ such that 
$L_g(f)=\big<g,f\big>$ holds ture for all $f\in h_{fin}^{p(\cdot),q}$ with
$\|g\|_{\widetilde{bmo}^{p(\cdot)}}\le C\|L\|$.
\end{thm}

\section{Boundedness of some operators}

In this section, we will consider the boundedness of
the inhomogeneous Calder\'{o}n-Zygmund singular integrals and the local fractional integrals.
First we recall the inhomogeneous Calder\'{o}n-Zygmund singular integrals in \cite{M1,DHZ}. Precisely, the operator $T$ is said to be an inhomogeneous Calder\'{o}n-Zygmund singular integral if  $T$ is a continuous linear operator from $ \mathcal{D}$ to $ \mathcal{D}'$ defined by
$$\langle T(f),g\rangle=\int\mathcal{K}(x,y)f(y)g(x)dxdy$$ for all $f,g\in\mathcal{D}(\mathbb{R}^{n})$ with disjoint supports, where $\mathcal K(x,y),$ the kernel of $T,$ satisfies the conditions as follows.
\begin{eqnarray*}
|\mathcal{K}(x,y)|\leq C\min\left\{\frac{1}{|x-y|^{n}},\frac{1}{|x-y|^{n+\delta}}\right\},\ \hbox{for some }\delta>0 \ \hbox{and}\ \ x\neq y.
\end{eqnarray*}
and for $\epsilon\in(0,1)$
\begin{eqnarray*}
|\mathcal{K}(x,y)-\mathcal{K}(x,y')|+
|\mathcal{K}(y,x)-\mathcal{K}(y',x)|\leq 
C\frac{|y-y'|^{\epsilon}}{|x-y|^{n+\epsilon}},
\end{eqnarray*}
when $|y-y'|\le \frac{1}{2}|x-y|$.

The first result of this section is the following
\begin{thm}\label{s6th1}
Suppose that $p(\cdot)\in LH$ and $\{\frac{n}{n+\varepsilon}\vee\frac{n}{n+\delta}\}<p^-\le p^+<\infty.$ Let $T$ be an inhomogeneous Calder\'{o}n-Zygmund singular integral. If $T$ is a
bounded operator on $L^{2}$, then 
$T$ can be extended to an
$(h^{p(\cdot)}-L^{p(\cdot)})$ bounded operator.  
That is, there exists a constant $C$ such that
\begin{align*}
\|T(f)\|_{L^{p(\cdot)}}\le C\|f\|_{h^{p(\cdot)}}.
\end{align*}
\end{thm}

\begin{proof}
Recalling the atomic decomposition of local Hardy space
$h^{p(\cdot)}$ in
Theorem \ref{s4th2}, we know that
if $f\in h^{p(\cdot)}$, then there exists non-negative coefficients  $\{\lambda_j\}_{j=1}^\infty$ and the 
special local $(p(\cdot),q)$-atoms $\{a_j\}_{j=1}^\infty$ 
such that
$f=\sum_{j=1}^\infty \lambda_ja_j$ in $h^{p(\cdot)}\cap L^q$, and that
\begin{equation*}
\left\|\sum_{j=1}^\infty\lambda_j\chi_{Q_j}\right\|_{L^{p(\cdot)}}
\le C\|f\|_{h^{p(\cdot)}}
\end{equation*}
for $0<q<\infty$. 
Then for $x\in\mathbb R^n$, we have
\begin{align*}
|T(f)(x)|&\le
\sum_{j}|\lambda_{j}|T(a_{j})(x)|
\chi_{\tilde Q_{j}}+\sum_{j}|\lambda_{j}||T(a_{j})(x)|
\chi_{\tilde Q_{j}^{c}}(x)\\
&=:I+II.
\end{align*}
To prove the theorem, it will
suffice to prove that
\begin{equation*}
\|T(f)\|_{L^{p(\cdot)}}\le
C\left\|\sum_{j=1}^\infty\lambda_j\chi_{Q_j}\right\|_{L^{p(\cdot)}}
\end{equation*}

First we need prove that
\begin{equation*}
\|I\|_{L^{p(\cdot)}}\le
C\left\|\sum_{j=1}^\infty\lambda_j\chi_{Q_j}\right\|_{L^{p(\cdot)}}.
\end{equation*}

Observe that $\mathcal K(x,y),$ the kernel of $T,$ satisfies the following conditions:
\begin{eqnarray*}
|\mathcal{K}(x,y)|\leq C\min\left\{\frac{1}{|x-y|^{n}},\frac{1}{|x-y|^{n+\delta}}\right\}
\le C\frac{1}{|x-y|^{n}},\ \hbox{for some }\delta>0 \ \hbox{and}\ \ x\neq y.
\end{eqnarray*}
and for $\epsilon\in(0,1)$
\begin{eqnarray*}
|\mathcal{K}(x,y)-\mathcal{K}(x,y')|+
|\mathcal{K}(y,x)-\mathcal{K}(y',x)|\leq 
C\frac{|y-y'|^{\epsilon}}{|x-y|^{n+\epsilon}},
\end{eqnarray*}
when $|y-y'|\le \frac{1}{2}|x-y|$.
Also, $T$ is a bounded operator on $L^{2}$.
From the Calder\'on-Zygmund real method
in \cite[Section 7.3]{M1}, we get
that $T$ is also bounded on $L^q$ for any $1<q<\infty$. 
Fix atoms $a_{j}$
supported in cubes $Q_{j}$.
For any $(p^+\vee1)<q<\infty$.
Then we have
\begin{align*}
  \left(\frac{1}{|Q_{j}|}\int_{Q_{j}}
  |T(a_{j})(x)|^{q}dx\right)^{1/{q}}
  \le \frac{1}{|Q_{j}|^{1/{q}}}\|a_{j}\|_{L^{q}}\le C.
\end{align*}

Applying Lemma \ref{s2l6}, we get that
\begin{align*}
\begin{split}
\|I\|_{L^{p(\cdot)}}
&\le\left\|\sum_{j}|\lambda_{j}||T(a_{j})|
\chi_{Q_{j}^\ast}\right\|_{L^{p(\cdot)}}\\
&\le C\left\|\sum_{j}|\lambda_{j}|\left(\frac{1}{|Q_{j}|}\int_{Q_{j}}
  |T(a_{j})|^{q}dx\right)^{1/{q}}\chi_{Q_{j}^\ast}\right\|_{L^{p(\cdot)}}\\
&\le C\left\|\sum_{j}\lambda_{j}
\chi_{Q_{j}^\ast}\right\|_{L^{p(\cdot)}}\le 
C\left\|\sum_{j}\lambda_{j}
\chi_{Q_{j}}\right\|_{L^{p(\cdot)}}.
\end{split}
\end{align*}
To estimate the term $II$,
we will divide in the following two case.

Case 1: $|Q_{j}|\le1$. In this case, $a_{j}$
satisfies the vanishing moment condition.
Noting that $x\in \tilde Q_j^c$ and $c_{Q_j}$
is the center of $Q_j$, 
we have $|x-c_{Q_j}|\ge 2|y-c_{Q_j}|$ and
$|y-c_{Q_j}|\le \ell(Q)$. By using the smooth condition of the kernel $\mathcal K$ we obtain that
\begin{align*}
  &|T(a_{j})(x)|
  =\left|\int_{Q_j}\mathcal K(x,y)a_{j}(y)dy\right|\\
\le& \int_{Q_j}|\mathcal K(x,y)-\mathcal K(x,c_{Q_j})|
|a_j(y)|d{y}\\
\le& C\int_{Q_j} \frac{|y-c_{Q_j}|^\epsilon}
{(|x-c_{Q_j}|)^{n+\epsilon}}|a_j(y)|dy
\\ \le &
C\|a_j\|_{L^\infty}
\frac{\ell(Q_{j})^{n+\epsilon}}{|x-c_{Q_j}|^{n+\epsilon}}\\\le&C
\left[\frac{\ell(Q_{j})^{n}}{|x-c_{Q_j}|^{n}}\right]^{\frac{n+\epsilon}{n}}
\sim (M(\chi_{Q_j})(x))^{\frac{n+\epsilon}{n}}.
\end{align*}

Case 2: $|Q_{j}|\ge1$.
In this case, we have $|x-y|\sim |x-c_{Q_j}|$ and 
$|x-y|\ge 1/2$ where $x\in \tilde Q_j^{c}$ and 
$y\in Q_j$.
By using the size condition of $K$, for any $x\in \tilde Q_j^{c}$ we obtain that
\begin{align*}
  &|T(a_{j})(x)|=\left|\int_{Q_j}K(x,y)a_{j}(y)dy\right|\\
\le & \int_{Q_j}|K(x,y)||a_j(y)|dy\\
\le & C\frac{|Q_j|}{|x-c_{Q_j}|^{n+\delta}}\\
\le & C\frac{\ell(Q_j)^{n+\delta}}{|x-c_{Q_j}|^{n+\delta}}\sim (M(\chi_{Q_j})(x))^{\frac{n+\delta}{n}}.
\end{align*}

We denote $\gamma=\{\frac{n+\epsilon}{n}\wedge\frac{n+\delta}{n}\}$. The condition
$\{\frac{n}{n+\varepsilon}\vee\frac{n}{n+\delta}\}<p^-$
means that $\gamma p^->1$.
Then we conclude that
\begin{align*}
&\|II\|_{L^{p(\cdot)}}
\le
C\Bigg\|\sum_{j}|\lambda_{j}|M^{\gamma}(\chi_{Q_j})\Bigg\|_{L^{p(\cdot)}}\\&\le
C\left\|\left(\sum_{j}
|\lambda_{j}|{M^{\gamma}(\chi_{Q_j})}\right)^{\frac{1}{\gamma}}
\right\|^\gamma_{L^{\gamma p(\cdot)}}
\leq C\left\|\sum_{j}\lambda_{j}
\chi_{Q_{j}}\right\|_{L^{p(\cdot)}}.
\end{align*}

Therefore, together with Remark \ref{s3r2} and a dense argument,
we finishes the proof of Theorem \ref{s6th1}.
\end{proof}

Next, we will establish mapping properties from variable local Hardy space $h^{p(\cdot)}$ into
itself for the inhomogeneous Calder\'{o}n-Zygmund singular integral $T$. To state it, we also need to assume one additional condition on $T$, $\int_{\mathbb R^n}T(a)(x)dx=0$ for the
special local $(p(\cdot),q)-$atoms $a$ and 
$\mbox{supp}\,a\subset Q$ with
$|Q|<1$. For convenience, we write $T^\ast_{loc}(1)=0$,
if $T$ satisfies the above moment condition.  
\begin{thm}\label{s6th2}
Suppose that $p(\cdot)\in LH$ and $\{\frac{n}{n+\varepsilon}\vee\frac{n}{n+\delta}\}<p^-\le p^+<\infty.$ Let $T$ be an inhomogeneous Calder\'{o}n-Zygmund singular integral. If $T$ is a
bounded operator on $L^{2}$ and $T^\ast_{loc}(1)=0$, then 
$T$ has a unique extension on $h^{p(\cdot)}$
and, moreover, there exists a constant $C$ such that
\begin{align*}
\|T(f)\|_{h^{p(\cdot)}}\le C\|f\|_{h^{p(\cdot)}},
\end{align*}
for all $f\in h^{p(\cdot)}$ .
\end{thm}

\begin{proof}
By the argument similar to that used in the above proof, it will suffice to prove that, for $h^{p(\cdot)}\cap L^q$ with $(p^+\vee1)<q<\infty$,
\begin{align*}
\|\mathcal G_N^0T(f)\|_{L^{p(\cdot)}}
\leq C\left\|\sum_{j}\lambda_{j}
\chi_{Q_{j}}\right\|_{L^{p(\cdot)}}.
\end{align*}

We claim that for $x\in\mathbb R^n$, we have
\begin{align*}
\sup_{0<t<1}|t^{-n}\psi(t^{-1}\cdot)\ast T(f)(x)|\le
\sum_{j}\lambda_{j}\left(M(T(a_{j})(x)\chi_{2\sqrt{n}\tilde Q_j}(x)+(M(\chi_{Q_{j}})(x))^{\gamma}\right),
\end{align*}
where $\gamma=\{\frac{n+\epsilon}{n}\wedge
\frac{n+\delta}{n}\}$.
Assuming the claim for the moment and repeating
the nearly identical argument to the proof of
Theorem \ref{s6th1}, we can obtain the desired result.
To this end, we only need to show the claim.
When $x\in 2\sqrt{n}\tilde Q_j$, we just need the pointwise estimate $\mathcal G_N^0 T(f)(x)\le C\sum_{j}\lambda_jM(T(a_j)(x))$.
When $x\in (2\sqrt{n}\tilde Q_j)^c$, we will consider two cases: $|Q_j|<1$ and $|Q_j\ge1$.

Consider the case $|Q_j|\ge1$. In this case,
$\ell(Q_j)\ge1$, then we have
\begin{align*}
|\psi_t\ast T(a_j)(x)|&=\left|\int_{\mathbb R^n} \psi_t(x-y)T(a_j)(y)dy\right|\\
&\le t^{-n}\int_{B(x,t)} |T(a_j)(y)|dy\\
&\le \sup_{y\in B(x,t)}|T(a_j)(y)|.
\end{align*}
Noting that $0<t<1\le\ell(Q_j)$ and  $x\in(2\sqrt{n}\tilde Q_j)^c$, it implies that $y\in(\tilde Q)^c$.
From the proof of Theorem \ref{s6th1}, we conclude that
\begin{align*}
  &\sup_{y\in B(x,t)}|T(a_{j})(y)|\le C(M(\chi_{Q_j})(x))^{\gamma}.
\end{align*}

Consider the case $|Q_j|<1$. 
If $0<t\le |x-c_{Q_j}|/2$, then together with $x\in (2\sqrt{n}\tilde Q_j)^c$ we can get that
$y\in (\tilde Q)^c$. Thus, repeating the same argument as used above, then we have
\begin{align*}
|\psi_t\ast T(a_j)(x)|\le C(M(\chi_{Q_j})(x))^{\gamma}.
\end{align*}
Finally, we consider the case that $t>|x-c_{Q_j}|/2$.
Observing that $\ell(Q_j)<1$, it follows that $a_j$ has the vanishing moment condition and $T_{loc}(1)=0$.
We write $\eta\equiv(\epsilon\wedge\delta)$.
For any $x\in (2\sqrt{n}\tilde Q_j)^c$, then by using the mean value theorem together with H\"older's inequality yield that
\begin{align*}
|\psi_t\ast T(a_j)(x)|&=
\left|\int_{\mathbb R^n} \psi_t(x-y)T(a_j)(y)dy\right|\\
&=\left|\int_{\mathbb R^n} \bigg(\psi_t(x-y)-\psi_t(x-c_{Q_j})\bigg)T(a_j)(y)dy\right|\\
&\le t^{-n}\int_{\mathbb R^n} \left|\frac{y-c_{Q-j}}{t}\right||\psi'((x-c_{Q_j}+\theta(c_{Q_j}-y))/t)||T(a_j)(y)|dy\\
&\le C|{x-c_{Q_j}}|^{-n-1}\int_{\mathbb R^n}
|y-c_{Q_j}||T(a_j)(y)|dy\\
&\le C|{x-c_{Q_j}}|^{-n-1}\left(\int_{\tilde Q_j}
|y-c_{Q_j}||T(a_j)(y)|dy+\int_{(\tilde Q_j)^c}
|y-c_{Q_j}||T(a_j)(y)|dy\right)\\
&\le C|{x-c_{Q_j}}|^{-n-1}\left(\ell(Q_j)^{\frac{n}{s'}+1}\|T(a_j)\|_{L^s}+\int_{(\tilde Q_j)^c}
\frac{\ell(Q_j)^{n+\eta}}{|y-c_{Q_j}|^{n+\eta-1}}dy\right)\\&\le C|{x-c_{Q_j}}|^{-n-1}\ell(Q_j)^{n+1}
\end{align*}
where $\theta\in (0,1)$ and $s\in(1,\infty)$.

Therefore, we complete the proof of Theorem \ref{s6th2}.
\end{proof}

We now show that the local fractional integrals are bounded from $h^{p(\cdot)}$ to $L^{q(\cdot)}$ when
$q^->1$, and from $h^{p(\cdot)}$ to $h^{q(\cdot)}$ when
$q^-\le1$. The following local fractional integral is introduced by D. Yang and S. Yang \cite{YY}.
\begin{defn}
Let $\alpha\in [0,n)$ and let $\phi_0\in\mathcal D$
be such that $\phi_0\equiv1$ on $Q(0,1)$ and 
$\mbox{supp}\,(\phi_0)\subset Q(0,2)$. The local
fractional integral $I_\alpha^{loc}(f)$ of $f$ is defined by
$$
I_\alpha^{loc}(f)(x)\equiv\int_{\mathbb R^n}
\frac{\phi_0(y)}{|y|^{n-\alpha}}f(x-y)dy.
$$
\end{defn}

\begin{thm}\label{s6th3}
Let $n$ be the non-negative integers
and $0<\alpha<n$. Suppose that $p(\cdot)\in LH\cap \mathcal P^0$, and
$\frac{1}{q(x)}=\frac{1}{p(x)}-\frac{\alpha}{n}$
for any $x\in\mathbb R^n$.
Then $I^{loc}_\alpha$ admits a bounded extension from
$h^{p(\cdot)}$  to $L^{q(\cdot)}$ when when $1<q^-\le q^+<\infty$.
Furthermore, when $0<q^-\le q^+\le 1$, $I^{loc}_\alpha$ admits a bounded extension from
$h^{p(\cdot)}$  to $h^{q(\cdot)}$ .
\end{thm}

\begin{proof}
The proof of this theorem is similar to the proof of Theorems \ref{s6th1} and \ref{s6th2} and so we only need to concentrate on the differences.
First we consider the case when $1<q^-\le q^+<\infty$.
By the atomic decomposition of $h^{p(\cdot)}$ and a dense argument, in order to show $I^{loc}_\alpha$ admits a bounded extension from
$h^{p(\cdot)}$  to $L^{q(\cdot)}$,
we only need to prove that
$$
\bigg\|\sum_{j}\lambda_jI^{loc}_\alpha (a_j)\bigg\|_{L^{q(\cdot)}}\leq C\left\|\sum_{j}\lambda_{j}
\chi_{Q_{j}}\right\|_{L^{p(\cdot)}}.
$$
To prove it, we will condider two cases for $\ell(Q)$.

Case 1: $\ell(Q_j)> 1$.
In this case, from the definition of $I_\alpha^{loc}(a_j)$,
we  know that $$\mbox{supp}\,(I_\alpha^{loc}(a_j))
\subset Q_j(c_{Q_j}, \ell(Q_j)+4)\subset 10\ell(Q_j).$$
Then we have 
\begin{align*}
\begin{split}
\bigg\|\sum_{j}\lambda_jI^{loc}_\alpha (a_j)\bigg\|_{L^{q(\cdot)}}
&=\left\|\sum_{j}\lambda_{j}|I^{loc}_\alpha (a_j)|
\chi_{10Q_{j}}\right\|_{L^{q(\cdot)}}\\
&\le C\left\|\sum_{j}\lambda_{j}\left(\frac{1}{|Q_{j}|}\int_{Q_{j}}
  |I^{loc}_\alpha (a_j)|^{q}dx\right)^{1/{q}}\chi_{10 Q_{j}}\right\|_{L^{q(\cdot)}}\\
&\le C\left\|\sum_{j}\lambda_{j}\ell(Q_j)^\alpha
\chi_{10Q_{j}t}\right\|_{L^{q(\cdot)}}\le 
C\left\|\sum_{j}\lambda_{j}
\chi_{Q_{j}}\right\|_{L^{p(\cdot)}},
\end{split}
\end{align*}
where the second inequality follows from the boundedness of $I_\alpha^{loc}$ on classical Lebesgue spaces
(\cite[Lemma 8.9]{YY}) and the last inequality follows from \cite[Lemma 5.2]{S}.

Case 2: $\ell(Q_j)\le1$. For any special $(p(\cdot),q)-$atom $a_j$,
we have the following pointwise estimates:
\begin{align*}
|I^{loc}_\alpha (a_j)(x)|
\le C\frac{\ell(Q_j)^{n+d+1}}{(\ell(Q_j)+|x-c_{Q_j}|)^{n+d+1-\alpha}}\le C\ell(Q_j)^\alpha (M\chi_{Q_j}(x))^{r},
\end{align*}
where $r=\frac{n+d+1-\alpha}{n+1}$.
In fact, when $|x-c_{Q_j}|\le \ell(\tilde Q_j)$,
by using the size condition of $a_j$,
we obtain that
$$
|I_\alpha^{loc}(a_j)(x)|\le C\int_{Q_j}
\frac{1}{|x-y|^{n-\alpha}}|a_j(y)|dy\le C\ell(Q_j)^\alpha.
$$
Let $P_N(y)$ be the Taylor polynomial of degree $d$
of the kernel of $I_\alpha^{loc}$ centered at $c_{Q_j}$. When $|x-c_{Q_j}|> \ell(\tilde Q_j)$, by using the moment condition of $a_j$ and the Taylor
expansion theorem, we have  
\begin{align*}
|I_\alpha^{loc}(a_j)(x)|&\le 
C\int_{Q_j}
\left|\frac{1}{|x-y|^{n-\alpha}}-P_N(y)\right||a_j(y)|dy\\
&\le C\int_{ Q_j}
\frac{1}{|x-c_{Q_j}|^{n+d+1-\alpha}}|y-c_{Q_j}|^{d+1}dy\\
&\le C
\frac{{\ell(Q_j)}^{n+d+1}}{|x-c_{Q_j}|^{n+d+1-\alpha}}.
\end{align*}

Similarly, we have
\begin{align*}
\begin{split}
\bigg\|\sum_{j}\lambda_jI^{loc}_\alpha (a_j)\bigg\|_{L^{q(\cdot)}}
&\le C\left\|\sum_{j}\lambda_{j}\ell(Q_j)^\alpha (M\chi_{Q_j})^{r}\right\|_{L^{q(\cdot)}}\\
&\le C\left\|\sum_{j}\lambda_{j}\ell(Q_j)^\alpha
\chi_{Q_{j}}\right\|_{L^{q(\cdot)}}\le 
C\left\|\sum_{j}\lambda_{j}
\chi_{Q_{j}}\right\|_{L^{p(\cdot)}}.
\end{split}
\end{align*}

Thus, we have proved the first part of the theorem.
Now we consider the boundedness of
$I_\alpha^{loc}$ from $h^{p(\cdot)}$ to $h^{q(\cdot)}$. To end this, we need to prove that
$$
\bigg\|\sum_{j}\lambda_j\mathcal G_N^0(I^{loc}_\alpha (a_j))\bigg\|_{L^{q(\cdot)}}\leq C\left\|\sum_{j}\lambda_{j}
\chi_{Q_{j}}\right\|_{L^{p(\cdot)}}.
$$
Observe that when $\ell(Q_j)>1$, 
from the definition of $\mathcal G_N^0(I_\alpha^{loc}(a_j))$, we see that $$\mbox{supp}\,(I_\alpha^{loc}(a_j))
\subset Q_j(c_{Q_j}, \ell(Q_j)+8)\subset 20\ell(Q_j),$$
which yields that
$
\|I^{loc}_\alpha (f)\|_{L^{q(\cdot)}}\leq C\left\|\sum_{j}\lambda_{j}
\chi_{Q_{j}}\right\|_{L^{p(\cdot)}}.
$
When $\ell(Q_j)\le 1$,
we follow the same proof as in \cite[Theorem 1.5]{CMN}
(also see \cite[Proposition 3.1]{T21}),
then we can obtain the desired results.

Therefore,  the proof of Theorem \ref{s6th3} is complete.
\end{proof}

Finally, we present the boundedness of the inhomogeneous Calder\'on-Zygmund operators and
the local fractional integrals on the
duals of $h^{p(\cdot)}$.
Precisely, we can obtain that the inhomogeneous Calder\'on-Zygmund operators is bounded on
$bmo^{p(\cdot)}$
and
the local fractional integrals is bounded from
$bmo^{q(\cdot)}$ to
$bmo^{p(\cdot)}$ and bounded from 
$L^{q(\cdot)'}$ to $bmo^{p(\cdot)}$ under some conditions.

By Theorems \ref{s6th2}, \ref{s6th3}, Corollary \ref{s5c1} and
\cite[Proposition 3.2]{TWL}, we have the following results.
\begin{cor}\label{s6c1}
Suppose that $p(\cdot)\in LH$ and $\{\frac{n}{n+\varepsilon}\vee\frac{n}{n+\delta}\}<p^-\le p^+<\infty.$ Let $T$ be an inhomogeneous Calder\'{o}n-Zygmund singular integral. If $T$ is a
bounded operator on $L^{2}$ and $T_{loc}(1)=0$, then 
there exists a constant $C$ such that
\begin{align*}
\|T(f)\|_{\widetilde{bmo}^{p(\cdot)}}\le C\|f\|_{\widetilde{bmo}^{p(\cdot)}},
\end{align*}
for all $f\in \widetilde{bmo}^{p(\cdot)}$ .
\end{cor}

\begin{cor}\label{s6c2}
Let $n$ be the non-negative integers
and $0<\alpha<n$. Suppose that $p(\cdot)\in LH$,
$0<p^-\le p^+\le 1$ and
$\frac{1}{q(x)}=\frac{1}{p(x)}-\frac{\alpha}{n}$
for any $x\in\mathbb R^n$.
For $0<q^-\le q^+\le1$, then
there exists a constant $C$ such that
\begin{align*}
\|I^{loc}_\alpha(f)\|_{bmo^{p(\cdot)}}\le C\|f\|_{bmo^{q(\cdot)}},
\end{align*}
for all $f\in bmo^{q(\cdot)}$ .
For $p^+<1<q^-\le q^+<\infty$, 
there exists a constant $C$ such that
\begin{align*}
\|I^{loc}_\alpha(f)\|_{bmo^{p(\cdot)}}\le C\|f\|_{L^{q'(\cdot)}},
\end{align*}
for all $f\in L^{q'(\cdot)}$ .
\end{cor}

\section*{Acknowledgements}
%The author wish to express his heartfelt thanks to the anonymous reviewer for corrections and so valuable suggestions.
The project is
sponsored by National Natural Science Foundation of China (No. 11901309), Natural Science Foundation of Jiangsu Province of China (No. BK20180734), Natural Science Research of Jiangsu Higher Education Institutions of China (No. 18KJB110022) and Nanjing University of Posts and Telecommunications Science Foundation (Nos. NY219114, NY217151).

\bibliographystyle{amsplain}

\end{document}